\documentclass[a4paper,11pt,reqno]{article}
\usepackage{latexsym, amsmath, amsfonts, amssymb, amsthm}
\usepackage{url}
\usepackage{graphics,subfigure}
\usepackage{mathtools}

\usepackage[utf8]{inputenc}
\usepackage[T1]{fontenc}
\usepackage{microtype}
\usepackage{enumitem}
\usepackage[scr=rsfs]{mathalfa}

\usepackage[font=small, labelfont={sc}]{caption}

\usepackage[a4paper,scale={0.73,0.75},marginratio={1:1},footskip=10mm,headsep=10mm]{geometry}

\usepackage{color}

\usepackage{hyperref}

\usepackage{cite}

\newcommand{\convas}[1][]{\xrightarrow[#1]{\mathrm{a.s.}}}
\newcommand{\convp}[1][]{\xrightarrow[#1]{\mathrm{p}}}
\newcommand{\convl}[1][]{\xrightarrow[#1]{\mathrm{law}}}
\newcommand{\convsl}[1][]{\xrightarrow[#1]{\mathrm{sl}}}

\newcommand{\eqlaw}{\stackrel{\mathrm{law}}{=}}

\numberwithin{equation}{section}

\newcommand{\vd}{\,\mathrm{d}}

\newtheorem{theorem}{Theorem}[section]
\newtheorem{itlemma}[theorem]{Lemma}
\newtheorem{itproposition}[theorem]{Proposition}
\newtheorem{itcorollary}[theorem]{Corollary}
\newtheorem{itremark}[theorem]{Remark}
\newtheorem{itdefinition}[theorem]{Definition}
\newtheorem{itexample}[theorem]{Example}
\newtheorem{itclaim}[theorem]{Claim}
\newtheorem{itfact}[theorem]{Fact}
\newtheorem{itassumption}[theorem]{Assumption}

\newenvironment{lemma}{\begin{itlemma}}{\end{itlemma}}
\newenvironment{remark}{\begin{itremark}\rm}{\end{itremark}}
\newenvironment{corollary}{\begin{itcorollary}}{\end{itcorollary}}

\makeatletter
\newtheorem*{rep@theorem}{\rep@title} \newcommand{\newreptheorem}[2]{%
\newenvironment{rep#1}[1]{%
\def\rep@title{\bf #2 \ref{##1} }%
\begin{rep@theorem} }%
{\end{rep@theorem} } }
\makeatother
\newreptheorem{theorem}{Theorem}
\newreptheorem{lemma}{Lemma}

\DeclareMathOperator{\sgn}{sgn}
\DeclareMathOperator{\leb}{Leb}

\def \R {{\mathbb R}}

\def \N {{\mathbb N}}
\def \PR {\mathbb{P}}
\def \QR {\mathbb{Q}}
\def \E {\mathbb{E}}
\def \FF {{\mathcal{F}}}

\newcommand\given{\nonscript\:\delimsize\vert\nonscript\:\mathopen{}} 
\newcommand\SetSymbol[1][]{\nonscript\:#1\vert\nonscript\:\mathopen{}\allowbreak}
\DeclarePairedDelimiterX\Set[1]\{\}{\renewcommand\given{\SetSymbol[\delimsize]}#1}
\DeclarePairedDelimiterX\Prb[1][]{\renewcommand\given{\SetSymbol[\delimsize]}#1}

\def\ve{\varepsilon}
\def\s{\sigma}

\def\l{\lambda}
\def\L{\Lambda}

\def\G{\mathcal{G}}
\def\a{\alpha}
\def\th{\theta}
\def\b{\beta}

\def\cG{\mathcal{G}}

\def\cF{\mathcal{F}}

\newcommand{\Keywords}[1]{\par\noindent 
{\small{\em Keywords\/}: #1}}

\title{Statistical estimation of the Oscillating Brownian Motion}

\author{Antoine Lejay\thanks{
	Université de Lorraine,
IECL, UMR 7502, Vand\oe uvre-lès-Nancy, F-54500, France;\newline
CNRS, IECL, UMR 7502, Vand\oe uvre-lès-Nancy, F-54500, France;\newline
Inria,  Villers-lès-Nancy, F-54600, France\newline
email: \texttt{Antoine.Lejay@univ-lorraine.fr}
} \and Paolo Pigato\thanks{
	Université de Lorraine,
IECL, UMR 7502, Vand\oe uvre-lès-Nancy, F-54500, France;\newline
CNRS, IECL, UMR 7502, Vand\oe uvre-lès-Nancy, F-54500, France;\newline
Inria,  Villers-lès-Nancy, F-54600, France\newline
email: \texttt{Paolo.Pigato@inria.fr}
} }

\begin{document}

\date{\today}

\maketitle

\begin{abstract} 
We study the asymptotic behavior of estimators of a two-valued, discontinuous
diffusion coefficient in a Stochastic Differential Equation, called an Oscillating
Brownian Motion. Using the relation of the latter process with the Skew Brownian Motion, 
we propose two natural consistent estimators, which are variants of the integrated volatility
estimator and take the occupation times into account. We show the stable convergence of the renormalized errors' estimations 
toward some Gaussian mixture, possibly corrected by a term that depends on the local
time. These limits stem from the lack of ergodicity as well as the behavior of
the local time at zero of the process.  We test both estimators on simulated
processes, finding a complete agreement with the theoretical predictions.
\end{abstract}

\smallskip
\Keywords{Oscillating Brownian Motion, Gaussian mixture, local time, occupation time, Arcsine distribution,
Skew Brownian Motion}
%\classification{60H07}

%\tableofcontents

\section{Introduction}

Diffusion processes with discontinuous coefficients attract more and more
attention for simulation and modelling purposes (see references in \cite{lejay_2012}).
Many domains are actually concerned, such as  geophysics \cite{ramirez_2013}, 
population ecology \cite{cc,Cantrell:1999iv}, finance \cite{interestrate,liptonsepp},
\dots\ among others.  
From a theoretical point of view, diffusions with discontinuous coefficients 
are an instance of Stochastic Differential Equations (SDE) with local time ---~also 
called \emph{skew diffusion}~--- for which many results are contained in the work of J.-F. Le Gall \cite{legall}.

Estimation come together with simulation and modelling, as models 
need to be calibrated. This article deals with the parametric estimation of the coefficients 
of a one-dimensional SDE of type 
\begin{equation}
    \label{eq:1}
    Y_t=x+\int_0^t \sigma(Y_s)\vd W_s+\int_0^t b(Y_s)\vd s
\end{equation}
where $W$ is a Brownian motion and $x\mapsto \sigma(x)$ takes two values $\Set{\sigma_+,\sigma_-}$
according to the sign of $x$, when the process $Y$ is observed at discrete times
$iT/n$, $i=0,\dotsc,n$ up to a time~$T$.
In a first part, we consider that \eqref{eq:1} contains no drift ($b=0$).
The solution $Y$ to \eqref{eq:1} is called an \emph{Oscillating Brownian Motion} (OBM). 
This process was studied by J.~Keilson and J.A. Wellner in~\cite{kw} who give some 
of its main properties and close form expressions for its density and occupation time.

We provide two very simple estimators which generalize the 
\emph{integrated volatility} (or averaged squared)
estimator (in finance, $\sigma$ is called the \emph{volatility}
when $Y$ represents the logarithm of the price of an asset). 
For two processes $Z,Z'$, we set 
\begin{equation*}
    [Z,Z']^n_T:=\sum_{i=1}^n (Z_{i,n}-Z_{i-1,n})(Z'_{i,n}-Z'_{i-1,n})
    \quad\text{and}\quad
    \bar{Q}^n_T(Z,+):=\frac{1}{n}\sum_{i=1}^{n} \mathbf{1}_{Z_{i,n}\geq 0}
\end{equation*}
with $Z_{i,n}=Z_{iT/n}$.
The quantity $\bar{Q}^n_T(Z,+)$ is an approximation of the \emph{occupation time}
of $Z$ on the positive axis up to time $T$.
Our estimator $(\hat{\sigma}_+^n)^2$ of $\sigma_+^2$ is then $[Y^+,Y^+]/\bar{Q}^n_T(Y,+)$,
where $Y^+$ is the positive part of $Y$.
A similar estimator is defined for $\sigma_-^2$. 
Although an analytic form is known for the density, this estimator 
is simpler to implement than the Maximum Likelihood Estimator. Besides, 
it also applies when $b\not=0$,  while explicit expressions for the density
become cumbersome, at best \cite{lejay_2016,sm}.

We show first that $\hat{\sigma}^n_+$ is a consistent estimator of $\sigma_+$. 
Yet it is asymptotically biased. We also prove that $\sqrt{n}((\hat{\sigma}^n_+)^2-\sigma_+^2)$
converges stably to a mixture of Gaussian distributions (in which, unsurprisingly, the occupation
time of the positive side appears) plus an explicit term giving the bias.

When estimating $\s_+$, the actual size of the ``useful'' sample is
proportional to the occupation time of $\R^+$. Therefore,
a dependence on the occupation time is to be expected in any
reasonable estimator.  The law of the occupation time for the OBM follows an
arcsine type distribution, which generalizes the one of the Brownian motion. Since these laws
carry much mass close to the extremes, this amounts to say that many
trajectories of this process spend most of the time on the same side of $0$. 
Therefore, with a high probability, either $\sigma_+$ or $\sigma_-$ is only based
on few observations. This affects our central limit theorem as well 
as the quality of the estimation of $(\sigma_+,\sigma_-)$,
meaning that the limit law will not be a Gaussian one, as one would expect from
the approximation of quadratic variation, but a Gaussian mixture 
displaying heavy tails.

Another tool of importance in this framework, strictly connected with the occupation time, is 
the \emph{local time}. Given a stochastic process $Y$, its local time at a
point $x$, denoted by  $\Set{L^x_t(Y)}_{t\geq 0}$,
represents the time spent by $Y$ at $x$, properly re-scaled. It has a
fundamental role in the study of SDEs with discontinuous coefficients.  
Intuitively, the local time appears when dealing with discontinuous coefficients because 
it helps to quantify what happens locally at the discontinuity. A  Lamperti's type
transform applied with the help of the Itô-Tanaka formula shows that $Y$ 
is intimately related to the \emph{Skew Brownian Motion} (SBM, see \cite{harrison-shepp_1981,le}) $X$, the 
solution to the SDE
\begin{equation*}
X_t=x+W_t+\theta L_t^0(X),\ -1<\theta<1,
\end{equation*}
through a simple deterministic transform $X=\Phi(Y)$ \cite{etore,lejaymartinez}.

In the present paper, the local time plays an important role for two reasons. 
First, because we use the transform $X=\Phi(Y)$ to apply some convergence
results which extend to the SBM some results of J.~Jacod on the asymptotic 
behavior of quantities of type $n^{-1/2}\sum_{i=1}^n f(X_{i-1,n},X_{i,n})$. 
Second, because the local time itself appears in the limit of the above quantities. 
Actually, the asymptotic bias is related to the local time. 

We then provide a second simple estimator of $\sigma_\pm^2$, defined as $(m^n_\pm)^2=[Y^\pm,Y]^n_T/\bar{Q}^n_T(Y,\pm)$
which is also consistent. We show that $\sqrt{n}((m^n_\pm)^2-\sigma_\pm^2)$ converges stably to a Gaussian mixture.
The asymptotic bias observed in $\sqrt{n}((\hat{\sigma}^n_\pm)^2-\sigma^2_\pm)$
is removed through the introduction of the quadratic term 
$[Y^+,Y^-]$ which is related to the local time. The variance of the former limit is not larger than the one
of the latter.

In Corollary \ref{girsanov}, we also generalize  these convergence results in presence of a bounded drift term. 
We prove that the estimators mentioned above converge to analogous limit random variables,
depending on the occupation time of the SDE under study.
Unlike for the OBM, the limit law is not explicit for SDEs with general drift,
since the law of the occupation time is not know, if not in special cases
(see \textit{e.g.} \cite{watanabe,watanabe1,kasahara}).

The novelty of the paper lies in the treatment of a discontinuous diffusion coefficient. 
This implies a drastic difference with the case of regular coefficients, as the situation 
cannot be reduced to a Brownian one (the measure of the SBM being singular with respect to the one 
of the Brownian motion \cite{legall}). This explains the presence of an asymptotic bias 
for  $\hat{\sigma}^n_\pm$, which is removed by a correction (leading to  $m^n_\pm$) which 
involves only a fraction of order $n^{-1/2}$ of the observations.

Besides, the framework is not the one of ergodic processes, like for many estimators, but of null recurrent ones.  
On the last point, our study does not fit the situations considered \textit{e.g.}, in \cite{hopfner,hopfner2,ac}.

With respect to many estimators constructed for diffusion processes, the lack of
ergodicity of the process explains the presence of a mixed normal
distribution in the limit. For diffusions, asymptotic convergence involving a
mixture of normal distributions (with different type of limits) is already
observed in the works of F. Florens-Zmirou~\cite{florens-zmirou} for
non-parametric estimation, and of J.~Jacod~\cite{jp,j1}, from which we borrow and adapt the
general treatment. The core of our proof requires the adaptation to the SBM of some results on the convergence
toward the local time given in \cite{j1}.

%%%%%%%%%%%%%%%%%%%%%%%%%%%%%%%%%%%%%%%%%%%%%%%%%%%%%%%%%%%%%%%%%%%%%%
%%%%%%%%%%%%%%%%%%%%%%%%%%%%%%%%%%%%%%%%%%%%%%%%%%%%%%%%%%%%%%%%%%%%%%
%%%%%%%%%%%%%%%%%%%%%%%%%%%%%%%%%%%%%%%%%%%%%%%%%%%%%%%%%%%%%%%%%%%%%%
%%%%%%%%%%%%%%%%%%%%%%%%%%%%%%%%%%%%%%%%%%%%%%%%%%%%%%%%%%%%%%%%%%%%%%

\subsection*{Content of the paper}

In Section \ref{sectionOBM} we define the Oscillating Brownian Motion (OBM) and
recall some useful properties. In Section~\ref{results} we define our
estimators and state precisely the convergence theorems. 
These results are then proved is Section~\ref{proofs}.
In Section~\ref{sectionORW} we consider the Oscillating Random
Walk, a discrete process that can be used to construct the OBM, and study an
estimator on this discrete process. 
Section~\ref{numerics} is devoted to the implementation of the estimators of
the OBM, and contains numerical experiments showing the good
behavior of the estimators in practice.

\subsection*{Notations}

For notational convenience, we work on the time interval $[0,1]$. 
Our results can be extended to a general time interval via a space-time re-scaling (see Remark~\ref{scaling}).
Throughout the paper, we use the following notation for convergence of random variables:
$\convp$ in probability;\quad
$\convl$ in law;\quad
$\convsl$ stable in law;\quad
$\eqlaw$ denotes equality in law. 
The Lebesgue measure is written $\leb$.
The positive and negative parts of $x\in \R$ are denoted by $x^+=x\vee 0$, $x^-=(-x)\vee 0$.
For any continuous semimartingale $M$, we write $\langle M \rangle$ for its quadratic variation process.
For $y\in \R$ we define the (symmetric) \emph{local time} of $M$ at $y$ as the process $(L_t^y(M))_{t\in[0,1]}$, with
(See \cite[Corollary~VI.1.9, p.~227]{ry})
\begin{equation*}
L^y_t(M)=\lim_{\ve\downarrow 0} \frac{1}{2\ve}\int_0^t 1_{\{y-\ve\leq M_s \leq y+\ve\}} \vd\langle M \rangle_s
\text{ almost surely}.
\end{equation*}
When we do not specify $y$, we mean the local time at $0$: $L_t(M)=L_t^0(M)$.

For fixed $n\in\N$, we consider the discretization of $[0,1]$ given by $0,1/n,\dotsc,1$.
For any process $(M_t)_{t\in[0,1]}$, we write $M_{k,n}=M_{k/n}$.
For any processes $M,\bar{M}$ we also set the ``discrete bracket''
\[
[M,\bar{M}]^n_1=\sum_{k=1}^n (M_{k,n}-M_{k-1,n})(\bar{M}_{k,n}-\bar{M}_{k-1,n}).
\]
We also write $[M]^n_1=[M,M]^n_1$.

%%%%%%%%%%%%%%%%%%%%%%%%%%%%%%%%%%%%%%%%%%%%%%%%%%%%%%%%%%%%%%%%%%%%%%
\section{Oscillating Brownian Motion}\label{sectionOBM}

For two parameters $\s_+,\s_->0$, we define the diffusion coefficient $\s$ as follows:
\begin{equation}\label{sigma}
\s(y)=\s_+ \mathbf{1}(y\geq 0)+\s_- \mathbf{1}(y< 0),\ \forall y\in\mathbb{R}.
\end{equation}

Let $W$ be a Brownian motion with its (completed) natural filtration $(\cG_t)_{t\geq 0}$
on a probability space $(\Omega,\cF,\mathbb{P})$. From now on, we denote by $Y$ the unique strong solution 
to 
\begin{equation}\label{obm}
Y_t=Y_0+\int_0^t \s(Y_s)\vd W_s.
\end{equation}
Strong existence and uniqueness of $Y$ is granted by the results of \cite{legall}. 
Following the terminology of \cite{kw}, we call $Y$ an \emph{Oscillating Brownian Motion} (OBM)

We recall some known properties of $Y$, proved in \cite{kw}. 

For $(t,x,y)\in(0,1]\times \R^2$, let $p(t,x,y)$ be the density of $Y_t$ in $y$, with initial condition $Y_0=x$. 
In \cite{kw} explicit formulas for the transition density are given. In particular, when $\sgn x\neq \sgn y$,
\begin{equation*}
p(t,x,y)=
\begin{cases}
\dfrac{2\s_+}{\s_-(\s_+ + \s_-)} \dfrac{1}{\sqrt{2\pi t}} e^{-(\frac{x}{\s_+}-\frac{y}{\s_-})^2\frac{1}{2t}}
    &\text{ for }x>0,\ y<0 \\[5mm]
\dfrac{2\s_- }{\s_+ (\s_+ + \s_-)} \dfrac{1}{\sqrt{2\pi t}} e^{-(\frac{y}{\s_+}-\frac{x}{\s_-})^2\frac{1}{2t}}	    &\text{ for } x<0,\ y>0.
\end{cases}
\end{equation*}
Integrating the previous equations we obtain
\begin{equation}\label{propone}
\begin{split}
\PR(Y_t<0|Y_0>0)
&=\frac{2\s_+}{\s_++\s_-}\Phi\left(-Y_0/(\s_+\sqrt{t})\right),\\
\PR\Prb*{Y_t>0\given Y_0<0}&=\frac{2\s_-}{\s_-+\s_+}\Phi\left(Y_0/(\s_-\sqrt{t})\right),
\end{split}
\end{equation}
where $\Phi$ is the cumulative density function of the standard Gaussian.
The law of the occupation time of $\R^+$ is also computed in \cite{kw}.
Let the occupation time of $\R^+$ be defined as
\begin{equation}\label{Q}
Q^+_t=\leb \{s\leq t:Y_s\geq 0\}.
\end{equation}
The distribution of $Q^+_t$, with $Y_0=0$ a.s., is explicit. 
The scaling $Q_1^+\eqlaw Q^+_t/t$ holds, and 
\begin{equation}\label{ma}
\PR(Q_1^+\in \vd u)=\frac{1}{\pi} \frac{1}{\sqrt{u(1-u)}} \frac{\s_+/\s_-}{1-(1-(\s_+/\s_-)^2)u}\vd u \text{ for } 0<u<1.
\end{equation}
This generalizes the arcsine law for the occupation time of the Brownian Motion.
The occupation time $Q^-$ on $\R^-$ is easily computed from $Q^+$ since obviously,
$Q^-_t+Q^+_t=t$ for any~$t\geq 0$.

We introduce the process $(X_t)_{t\geq 0}$, setting $X_t= Y_t/\s(Y_t)$ for $t\geq 0$. 
It follows from the Itô-Tanaka formula that $X$ is a Skew Brownian Motion (SBM, see \cite{le}), meaning that
$X$ satisfies the following SDE:
\begin{equation}\label{sbm}
X_t=X_0+B_t+\theta L_t(X),
\end{equation}
where $B$ is a Brownian Motion, $L_t (X)$ is the symmetric local time of $X$ at $0$, $X_0=Y_0/\s(Y_0)$ and the coefficient $\theta$ is given by
\begin{equation}\label{theta}
\theta=\frac{\s_- -\s_+}{\s_- +\s_+}.
\end{equation}
We write from now on BM for Brownian Motion, SBM for Skew Brownian Motion, OBM for Oscillating Brownian Motion. 
The local times of $X$ and $Y$ are related by 
\begin{equation}\label{elt}
L_t(X)=\frac{\s_++\s_-}{2\s_+\s_-} L_t(Y)
\end{equation}
(see \cite{le} for a special case from which we easily recover this formula).

%%%%%%%%%%%%%%%%%%%%%%%%%%%%%%%%%%%%%%%%%%%%%%%%%%%%%%%%%%%%%%%%%%%%%%
\section{Main results}\label{results}

\subsection{The stable convergence}

Before stating our results, we need to recall the notion of \emph{stable convergence}, which was introduced by 
A.~Rényi \cite{renyi}. We refer to \cite{js} or \cite {jp} for a detailed exposition.
Let $Z_n$ a sequence of $E$-valued random variables defined on the same probability space $(\Omega,\FF,\PR)$. Let~$Z$ be an $E$-valued random variable defined on an extension, $(\tilde{\Omega},\tilde{\FF},\tilde{\PR})$.
We then say that $Z_n$ \emph{converges stably} to $Z$ (and write $Z_n\convsl[n\to\infty] Z$) if:
\[
\E[ Yf(Z_n)] \xrightarrow[n\to\infty]{} \tilde{\E} [Yf(Z)] 
\] 
for all bounded continuous functions $f$ on $E$ and all bounded random variables Y on $(\Omega,\mathcal{F})$ (or, equivalently, for all $Y$ as above and all functions $f$ which are bounded and Lipschitz). This notion of convergence is stronger than convergence in law, but weaker than convergence in probability.  We use in this paper the following crucial result: for random variables $Y_n$, $Z_n$ ($n\geq 1$), $Y$ and $Z$, 
\begin{equation*}\label{scl}
    \text{if }
Z_n \convsl[n\to\infty] Z
\text{ and }
Y_n \convp[n\to\infty] Y
\text{ then }
(Y_n,Z_n) \convsl[n\to\infty] (Y,Z).
\end{equation*}

%%%%%%%%%%%%%%%%%%%%%%%%%%%%%%%%%%%%%%%%%%%%%%%%%%%%%%%%%%%%%%%%%%%%%%
\subsection{Estimators for the parameters the Oscillating Brownian motion}

Let us assume that we observe the process $Y$ solution to \eqref{obm}
at the discrete times $0,1/n,\dotsc,1$. 
We want to estimate $\s_+,\s_-$ from these observations. 

A natural estimator for the occupation time $Q^+_1$ defined in \eqref{Q} 
is given by the  Riemann sums (see Section~\ref{secot}):
\begin{equation}\label{eot}
\bar{Q}^n_1(Y,+)= \sum_{k=1}^n \frac{ \mathbf{1}(Y_k\geq 0)}{n}.
\end{equation}
We define now $\hat{\s}^n_+$ as
\begin{equation}\label{estimator}
\hat{\s}^n_+= \sqrt{\frac{[Y^+]^n_1}{ \bar{Q}^n_1(Y,+)}},
\end{equation}
which we show to be a consistent estimator for $\s_+$.
Similarly, we set
\[
\bar{Q}^n_1(Y,-)= \sum_{k=1}^n \frac{ \mathbf{1}(Y_k< 0)}{n}=1-\bar{Q}^n_1(Y,+)
\text{ and }\hat{\s}^n_-= \sqrt{\frac{[Y^-]^n_1}{ \bar{Q}^n_1(Y,-)}},
\]
Finally, we define our estimator of the vector $(\sigma_+,\sigma_-)'$ as 
\begin{equation*}
\hat{\s}^n= \begin{pmatrix} \hat{\s}^n_+ \\
    \hat{\s}^n_-
\end{pmatrix}.
\end{equation*}

\begin{theorem}\label{maintheorem}
Let $Y$ be solution of \eqref{obm} with $Y_0=0$ a.s., and $\hat{\s}^n$ defined in \eqref{estimator}. 
Then
\begin{enumerate}[leftmargin=1em,label={{\textup{(\roman*)}}}]
\item
The estimator is consistent:
\[
\hat{\s}^n \convp[n\to\infty] \begin{pmatrix}\s_+\\ \s_-\end{pmatrix}.
\]
\item There exists an extension $(\tilde{\Omega},\tilde{F},\tilde{\mathbb{P}})$ of
    $(\Omega,\cF,\mathbb{P})$ carrying a Brownian motion $\bar{B}$ independent from $W$ such that 
\begin{equation}
    \label{eq:conv:1}
\sqrt{n}
\begin{pmatrix}
(\hat{\s}^n_+)^2 - \s_+^2\\
(\hat{\s}^n_-)^2 - \s_-^2
\end{pmatrix}
\convsl[n\to\infty]
\begin{pmatrix}
\dfrac{\sqrt{2}\s_+^2 }{Q^+_1}  \int_0^1  \mathbf{1}(Y_s > 0)\vd\bar{B}_s
\\
\dfrac{\sqrt{2}\s_-^2}{1-Q^+_1} \int_0^1  \mathbf{1}(Y_s < 0)\vd\bar{B}_s 
\end{pmatrix}
  - 
  \begin{pmatrix}
\frac{1 }{Q^+_1}  
\\
\frac{1 }{1-Q^+_1}  
\end{pmatrix}
 \frac{2\sqrt{2}}{3\sqrt{\pi}} \left(\frac{\s_- \s_+}{\s_++\s_-}\right) L_1(Y).
\end{equation}
The stable limit in the above result depends on the path of $Y$ through its local time $L_1(Y)$ and its occupation time $Q^\pm_1$.
By identifying the distribution of the limit, we can rewrite \eqref{eq:conv:1} as 
a convergence in distribution involving elementary random variables as 
\begin{multline}
    \label{eq:limit:biased}
\sqrt{n}\begin{pmatrix}
(\hat{\s}^n_+)^2 - \s_+^2\\
(\hat{\s}^n_-)^2 - \s_-^2
\end{pmatrix}
\convl[n\to\infty]
\begin{pmatrix}
\frac{\sqrt{2}\s_+^2 }{\sqrt{\L}} \mathcal{N}_1
\\
\frac{\sqrt{2}\s_-^2 }{\sqrt{1-\L}} \mathcal{N}_2
\end{pmatrix}
-
 \frac{8\sqrt{2}}{3\sqrt{\pi}} \frac{(\s_- \s_+)^2}{\s_++\s_-}
 \frac{\xi}{\sqrt{(1-\L)\s_-^2+\L\s_+^2 }}
 \begin{pmatrix}
\sqrt{\frac{1-\L}{\L}}
\\
\sqrt{\frac{\L}{1-\L}}
\end{pmatrix}
\\
=
\begin{pmatrix}
\frac{\sqrt{2}\s_+^2 }{\sqrt{\L}} 
\left(\mathcal{N}_1
-
 \frac{8}{3\sqrt{\pi}} \frac{1}{r+1}
 \frac{\xi \sqrt{1-\L}
}{\sqrt{(1-\L)+\L r^2 }}\right)
\\
\frac{\sqrt{2}\s_-^2 }{\sqrt{1-\L}} 
\left(\mathcal{N}_2
-
 \frac{8}{3\sqrt{\pi}} \frac{1}{1/r+1}
 \frac{\xi \sqrt{\L}
}{\sqrt{\L+(1-\L)/r^2 }}\right)
\end{pmatrix},
\end{multline}
where $r=\s_+/\s_-$, $\xi,\mathcal{N}_1,\mathcal{N}_2,\L$ are mutually independent, $\xi\sim \exp(1)$, $\mathcal{N}_1,\mathcal{N}_2\sim N(0,1)$ and $\L$ follows the modified arcsine law \eqref{ma} with density 
\[
p_\L(\tau)=\frac{1}{\pi\tau^{1/2}(1-\tau)^{1/2}}\frac{r}{1-(1-r^2)\tau}.
\]
\end{enumerate}
\end{theorem}

\begin{remark}
The Brownian case is $\s=:\s_+=\s_-$, $r=1$. The limit law is
\[
\sqrt{2}\s^2 
    \begin{pmatrix}
\frac{1}{\sqrt{\L}} 
\left(\mathcal{N}_1
-
 \frac{4}{3\sqrt{\pi}}
 \xi \sqrt{1-\L}
\right)
\\
\frac{1 }{\sqrt{1-\L}} 
\left(\mathcal{N}_2
-
 \frac{4}{3\sqrt{\pi}}
 \xi \sqrt{\L}
\right)
\end{pmatrix}
\]
where $\L$ follows the classical arcsine law (see \cite{levy,ry}).
\end{remark}

\begin{remark}
In \eqref{aprlt} we prove
\[
\sqrt{n}[Y^+,Y^-]^n_1\convsl[n\to\infty] \frac{2\sqrt{2}}{3\sqrt{\pi}}\frac{\s_+\s_-}{\s_++\s_-}L_1(Y).
\]
Actually, each term 
of type $(Y^+_{t_i}-Y^+_{t_{i-1}})(Y^-_{t_i}-Y^-_{t_{i-1}})$ vanishes unless 
$\sgn(Y_{t_i})\not=\sgn(Y_{t_{i-1}})$. 
Thus, $\sqrt{n}[Y^+,Y^-]^n_1$ provides us with as an estimator of the local time
since it somehow counts the number of crossings of zero (\textit{cf.} \cite{j1,lmt1}).
\end{remark}

\begin{remark}
We assume $Y_0=0$ a.s. because we need $Y_0$ to visit both $\R^+$ and $\R^-$.
This happens a.s. in any right neighborhood of $0$ if the diffusion starts from
$0$. If the initial condition is not $0$, we shall wait for the first time at
which the process reaches $0$, say $\tau_0$, and consider the (random) interval
$[\tau_0,T]$. 
\end{remark}

We define now a different estimator for $\s_\pm$ by 
\begin{equation}\label{estimator4}
m^n_+:= \sqrt{\frac{[Y^+,Y]^n_1}{ \bar{Q}^n_1(Y,+)}}
,\quad\quad
m^n_-:= \sqrt{\frac{[Y^-,Y]^n_1}{ \bar{Q}^n_1(Y,-)}}
\quad
\text{and}
\quad 
m^n:= \begin{pmatrix}
    m^n_+\\
    m^n_-
\end{pmatrix}.
\end{equation}

\begin{theorem}\label{maintheorem2}
Let $Y$ be solution of \eqref{obm} with $Y_0=0$ a.s., and $m^n$ defined in \eqref{estimator4}. 
The following convergence holds for $n\rightarrow \infty$:
\begin{equation*}
 \sqrt{n}
 \begin{pmatrix}
(m^n_+)^2 - \s_+^2\\
(m^n_-)^2 - \s_-^2
\end{pmatrix}
\convsl[n\to\infty]
\begin{pmatrix}
\frac{\sqrt{2}\s_+^2 }{Q^+_1}  \int_0^1  \mathbf{1}(Y_s > 0)\vd\bar{B}_s
\\
\frac{\sqrt{2}\s_-^2}{1-Q^+_1} \int_0^1  \mathbf{1}(Y_s < 0)\vd\bar{B}_s 
\end{pmatrix},
\end{equation*}
where $\bar{B}$ is a BM independent of $Y$ on an extension $(\tilde{\Omega},\tilde{\cF},\tilde{\mathbb{P}})$
of $(\Omega,\cF,\mathbb{P})$. We can rewrite such convergence as follows:
\begin{equation}
    \label{eq:limit:unbiased}
 \sqrt{n}
 \begin{pmatrix}
(m^n_+)^2 - \s_+^2\\
(m^n_-)^2 - \s_-^2
\end{pmatrix}
\convl[n\to\infty]
\begin{pmatrix}
\frac{\sqrt{2}\s_+^2 }{\sqrt{\L}} \mathcal{N}_1
\\
\frac{\sqrt{2}\s_-^2 }{\sqrt{1-\L}} \mathcal{N}_2
\end{pmatrix},
\end{equation}
where $\mathcal{N}_1,\mathcal{N}_2,\L$ are mutually independent, 
$\mathcal{N}_1,\mathcal{N}_2\sim N(0,1)$ and $\L$ follows the modified arcsine law with density given by \eqref{ma}, 
with $r=\s^+/\s_-$.
\end{theorem}
\begin{remark}
Comparing Theorems \ref{maintheorem} and \ref{maintheorem2}, we see that an
asymptotic bias is present in~$\hat{\s}^n$, but not in $m^n$. This bias has the
same order ($\sim 1/\sqrt{n})$ as the ``natural fluctuations'' of the
estimator. Because the local time is positive, it is more likely that $\hat{\s}^n_\pm$ 
underestimates~$\sigma_\pm$. For a more quantitative comparison between the convergence of
the two estimators, see Remark \ref{remaprlt}. In Section \ref{numerics}, 
we compare the two estimators in practice.
\end{remark}
\begin{remark}
    \label{scaling}
Theorem \ref{maintheorem2} gives  the asymptotic behavior for an estimator of $(\s_+,\s_-)$ in the presence of
high frequency data, yet with fixed time horizon $T=1$. The OBM enjoys a \emph{scaling property}:
if $Y$ is an OBM issued from $0$, then $(\sqrt{c}Y_{t/c})_{t\in \R_+}$ is an OBM issued form $0$, for any $c>0$ constant 
(see \cite[Exercise IX.1.17, p.~374]{ry}). Using this fact, we can easily generalize Theorem \ref{maintheorem2}
to the case of data on a time interval $[0,T]$ for some fixed $T>0$. We set
\begin{equation}
    \label{eq:mT}
m^{n,T}_+:= \sqrt{\frac{[Y^+,Y]^n_T}{ \bar{Q}^n_T(Y,+)}}
,\quad\quad
m^{n,T}_-:= \sqrt{\frac{[Y^-,Y]^n_T}{ \bar{Q}^n_T(Y,-)}}.
\end{equation}
The estimator is consistent and we have the following convergence:
\[
 \sqrt{n}
 \begin{pmatrix}
(m^{n,T}_+)^2 - \s_+^2\\
(m^{n,T}_-)^2 - \s_-^2
\end{pmatrix}
\convsl[n\to\infty]
\begin{pmatrix}
\frac{\sqrt{2T}\s_+^2 }{Q^+_T}  \int_0^T  \mathbf{1}(Y_s > 0)\vd\bar{B}_s
\\
\frac{\sqrt{2T}\s_-^2}{T-Q^+_T} \int_0^T  \mathbf{1}(Y_s < 0)\vd\bar{B}_s 
\end{pmatrix},
\]
where $\bar{B}$ is a BM independent of $W$ on an extension of the underlying probability space.
The limiting random variable follows the law given in \eqref{eq:limit:unbiased}, which actually does not depend on~$T$.

A slightly different approach is to imagine that our data are not in high frequency, but that we observe the process at regular time intervals, for a long time. In this case it is more reasonable to consider an OBM
$(Y_t)_{t\in \R_+}$, construct an estimator depending  on
$(Y_i)_{i =0,1,\dots,T-1, T}, \,T\in \N$, and then take the limit in long time. We set
\begin{equation*}
\mu^T_+:= \sqrt{\frac{[Y^+,Y]^T_T}{ \bar{Q}^T_T(Y,+)}}
,\quad\quad
\mu^T_-:= \sqrt{\frac{[Y^-,Y]^T_T}{ \bar{Q}^T_T(Y,-)}}.
\end{equation*}
Using again Theorem \ref{maintheorem2} and the diffusive scaling of the OBM, we have the following convergence:
\[
 \sqrt{T}
 \begin{pmatrix}
(\mu^T_+)^2 - \s_+^2\\
(\mu^T_-)^2 - \s_-^2
\end{pmatrix}
\convl[T\to\infty]
\begin{pmatrix}
\frac{\sqrt{2}\s_+^2 }{\sqrt{\L}} \mathcal{N}_1
\\
\frac{\sqrt{2}\s_-^2 }{\sqrt{1-\L}} \mathcal{N}_2
\end{pmatrix}.
\]
The limit distribution is again the law given in \eqref{eq:limit:unbiased}.
Theorem \eqref{maintheorem} can also be generalized to high frequency data on
an interval $[0,T]$ and to equally spaced data in long time, using the
diffusive scaling and \eqref{j1}-\eqref{j2}. For example, analogously to \eqref{eq:mT}, we define
 \[
\hat{\s}^{n,T}_+:= \sqrt{\frac{[Y^+,Y^+]^n_T}{ \bar{Q}^n_T(Y,+)}}
,\quad\quad
\hat{\s}^{n,T}_-:= \sqrt{\frac{[Y^-,Y^-]^n_T}{ \bar{Q}^n_T(Y,-)}}.
\]
Again, the limit law does not change and is the one given in \eqref{eq:limit:biased}.
\end{remark}
%%%%%%%%%%%%%%%%%%%%%%%%%%%%%%%%%%%%%%%%%%%%%%%%%%%%%%%%%%%%%%%%%%%%%%

\subsection{A generalization to OBM with drift}

We consider now a wider class of processes, adding a drift term to equation \eqref{obm}. Formally, 
let now $\xi$ be the strong solution to 
\begin{equation}\label{eqdrift}
d\xi_t=b(\xi_t)\vd t+\s(\xi_t)\vd W_t,
\end{equation}
with $\xi_0=0$, $\s$ defined in \eqref{sigma} and $b$ measurable and bounded. 
Again, strong existence and uniqueness of the solution to \eqref{eqdrift} is ensured by the results
of \cite{legall}. 

 Let $m^n(\xi)$ be defined as in \eqref{estimator4}:
\[
m^n_+(\xi)= \sqrt{\frac{[\xi^+,\xi]^n_1}{ \bar{Q}^n_1(\xi,+)}}
,\quad 
m^n_-(\xi)= \sqrt{\frac{[\xi^-,\xi]^n_1}{ \bar{Q}^n_1(\xi,-)}}
\quad
\text{and}
\quad
m^n(\xi):= \begin{pmatrix}
    m^n_+(\xi)\\
    m^n_-(\xi)
\end{pmatrix}.
\]
Let us also denote $Q^+_1(\xi)=\int_0^1 \mathbf{1}(\xi_s > 0) \vd s $.

\begin{corollary}\label{girsanov}
The following convergence holds for $n\rightarrow \infty$:
\[
 \sqrt{n}
 \begin{pmatrix}
(m^n_+(\xi))^2 - \s_+^2\\
(m^n_-(\xi))^2 - \s_-^2
\end{pmatrix}
\convsl[n\to\infty]
\begin{pmatrix}
\frac{\sqrt{2}\s_+^2 }{Q^+_1(\xi)}  \int_0^1  \mathbf{1}(\xi_s > 0)\vd\bar{B}_s
\\
\frac{\sqrt{2}\s_-^2}{1-Q^+_1(\xi)} \int_0^1  \mathbf{1}(\xi_s < 0)\vd\bar{B}_s 
\end{pmatrix},
\]
where $\bar{B}$ is a BM independent of $W$ on an extension of the underlying probability space.
 We can rewrite such convergence as follows:
\[
 \sqrt{n}
 \begin{pmatrix}
(m^n_+(\xi))^2 - \s_+^2\\
(m^n_-(\xi))^2 - \s_-^2
\end{pmatrix}
\convl[n\to\infty]
\begin{pmatrix}
\frac{\sqrt{2}\s_+^2 }{\sqrt{\Theta}} \mathcal{N}_1
\\
\frac{\sqrt{2}\s_-^2 }{\sqrt{1-\Theta}} \mathcal{N}_2
\end{pmatrix},
\]
where $\mathcal{N}_1,\mathcal{N}_2,\Theta$ are mutually independent, $\mathcal{N}_1,\mathcal{N}_2\sim N(0,1)$ and $\Theta\eqlaw Q^+_1(\xi)$.
\end{corollary}

\begin{remark}
Unlike for the OBM, the limit law is not explicit in Corollary~\ref{girsanov}, 
since the law of the occupation time of the positive axes is not
know in general (See \textit{e.g.}, \cite{lamperti,watanabe,kasahara,watanabe1}). 
On the other hand, some information on the law
of $\Theta$ can be obtained, at least in some special cases, via Laplace
transform. 

We also stress that this dependence on the occupation time is due the actual
sample size of the data giving us useful information. Indeed, when estimating
$\s_+$, the number of intervals that we can use is proportional to the
occupation time of $\R^+$. Analogously for the negative part.
\end{remark}

\begin{remark}
Actually, Corollary \ref{girsanov} holds under weaker assumptions. 
An example of model fitting into this framework is the SET-Vasicek model \cite{interestrate},
a generalization of the Vasicek interest rate model to a non-constant volatility, given exactly by \eqref{sigma}:
\[
    \vd \xi_t= -\a(\xi_t-\b)\vd t+\s(\xi_t)\vd W_t.
\]
\end{remark}

\begin{remark}\label{rem:scaling-drift}
The scaling property described in Remark~\ref{scaling} no longer holds in this situation,
so that the estimator can only be used in the ``high frequency'' setting. 
\end{remark}

%%%%%%%%%%%%%%%%%%%%%%%%%%%%%%%%%%%%%%%%%%%%%%%%%%%%%%%%%%%%%%%%%%%%%%
%%%%%%%%%%%%%%%%%%%%%%%%%%%%%%%%%%%%%%%%%%%%%%%%%%%%%%%%%%%%%%%%%%%%%%
%%%%%%%%%%%%%%%%%%%%%%%%%%%%%%%%%%%%%%%%%%%%%%%%%%%%%%%%%%%%%%%%%%%%%%

\section{Proofs of the convergence theorem}
\label{proofs}

This section is devoted to the proof Theorems \ref{maintheorem} and \ref{maintheorem2}. 
We first deal with some general approximation results which are well known for
diffusions with regular coefficients (see \cite{j1,jp,js}), but not for 
the framework considered here with discontinuous coefficients (when $\theta\not=0$, 
the law of the SBM is singular with respect to the one of the BM \cite{legall}).

Following \cite{lejaymartinez,etore}, 
we use the connection between the OBM and the SBM through a Lamperti-type transform.
Hence, we apply the results of \cite{lmt1} to the convergence of estimators of quadratic variation,
covariation and occupation time for these processes. Finally, we use all these
results to prove the main Theorems \ref{maintheorem} and \ref{maintheorem2}.

%%%%%%%%%%%%%%%%%%%%%%%%%%%%%%%%%%%%%%%%%%%%%%%%%%%%%%%%%%%%%%%%%%%%%%

\subsection{Approximation results}

Let us write
\[
\leb(\phi) = \int_{-\infty}^\infty \phi(\a) \vd\a
\]
for the Lebesgue integral of a function. 
In \cite{lmt1},  the following approximation result, borrowed from \cite{j1}, 
is proved for the SBM solution of \eqref{sbm}.

\begin{lemma}
\label{convskew}
Let $f$ be a bounded function such that $\int |x|^k |f(x)|\vd x<\infty$ for $k=0,1,2$ and~$X$ be a SBM
of parameter $\theta\in[-1,1]$ (\textit{i.e.}, the solution to \eqref{sbm}). 
Then for any $a>0$,
\[
\PR \Prb*{\left|\frac{1}{\sqrt{n}} \sum_{i=1}^{n-1} f(X_{i,n}\sqrt{n})- \bar{\l}_\th(f) L_1(X)\right|>a }
    \rightarrow 0,
\]
where $\Set{L_t(X)}_{t\geq 0}$ is the symmetric local time at $0$ of the SBM and
\begin{equation}\label{la1}
\bar{\l}_\th(f)=(1+\th)\leb(f^+)+(1-\th)\leb(f^-).
\end{equation}
\end{lemma}

\begin{remark}\label{j41}
In particular, when $\th=0$, $X$ is BM and the coefficient in front of the local time is simply $\leb(f)$. 
We recover there a special case of
a theorem by J.~Jacod (see \cite{j1}, Theorem 4.1). 
\end{remark}

We prove now an approximation result for the OBM.
\begin{lemma}
\label{convY}
Let $Y$ be the OBM in \eqref{obm}. 
Let $f$ be a bounded function such that $\int |x|^k |f(x)|\vd x<\infty$ for $k=0,1,2$.
Then for any $a>0$,
\[
\PR \Prb*{\left|\frac{1}{\sqrt{n}} \sum_{i=1}^{n-1} f(Y_{i,n}\sqrt{n})- 
\l_\s(f)
 L_1(Y)\right|>a }\rightarrow 0,
\]
where $L_t(Y)$ is the local time of $Y$ and
\begin{equation*}\label{la}
\l_\s(f)=2\left(\frac{\leb(f^+)}{\s_+^2} +
\frac{\leb(f^-)}{\s_-^2}\right).
\end{equation*}
\end{lemma}

\begin{proof}
Recall $Y_t/\s(Y_t)=X_t$, \eqref{elt} and \eqref{theta}. Let $\tilde{f}(x)=f(\s(x)x) $.
We have
\begin{multline*}
\bar{\l}_\th(\tilde{f})=(1+\th)\leb(\tilde{f}^+)+(1-\th)\leb(\tilde{f}^-)
\\
=\frac{2\s_-}{\s_+(\s_-+\s_+)}\leb(f^+)+
\frac{2\s_+}{\s_-(\s_-+\s_+)}\leb(f^-)
= \frac{2\s_+\s_-}{(\s_-+\s_+)} \l_\s(f),
\end{multline*}
so $\l_\s(f) L_t(Y)=\bar{\l}_\th(\tilde{f}) L_t(X)$. Therefore, from Lemma \ref{convskew},
for any $a>0$, 
\begin{multline*}
\PR \Prb*{\left|\frac{1}{\sqrt{n}} \sum_{i=1}^{n-1} f(Y_{i,n}\sqrt{n})- 
\l_\s(f)
 L_t(Y)\right|>a }\\
 =
\PR \Prb*{\left|\frac{1}{\sqrt{n}} \sum_{i=1}^{n-1}\tilde{f}( X_{i,n}\sqrt{n})- 
\bar{\l}_\th(\tilde{f}) L_1(X)\right|>a } \xrightarrow[n\to\infty]{} 0.
 \end{multline*}
This concludes the proof.
\end{proof}

We state now a very special case of Theorem~3.2 in  \cite{j2}, that we apply
several times in this work. The version in \cite{j2} holds for semimartingales,
not only martingales, the processes involved can be multi-dimensional, and the
limit process is not necessarily $0$. Anyways, we do not need this general
framework here. Stating the theorem only for one-dimensional martingales converging
to $0$ allows us to keep a simpler notation, which we introduce now:
for each c\`adl\`ag process $J$ we write $\Delta_i^n J=J_{i/n}-J_{(i-1)/n}$.
Consider a filtered probability space $(\Omega, \mathcal{F},\mathscr{F}, \PR)$
carrying a Brownian motion $B$. The filtration $\mathscr{F}=(\mathcal{F}_t)_{t\in [0,1]}$
is the natural (completed) one  for the Brownian motion. 
We define the filtration $\mathscr{F}^n$ as the ``discretization'' defined by $\mathcal{F}^n_t=\mathcal{F}_{[nt]/n}$.
We consider a $\mathscr{F}^n$-martingale in $\R$, \textit{i.e.}, a process of the form 
\[
Z^n_1=\sum_{i=1}^n \chi_i^n,
\]
where each $\chi_i^n$ is $\mathcal{F}_{i/n}$ measurable, square-integrable,
and $\E\Prb{ \chi_i^n \given \mathcal{F}_{\frac{i-1}{n}} }=0$. 

\begin{theorem}[{{Simplified form of Theorem~3.2 in \cite{j2}}}]
    \label{jm}
Suppose that
\begin{gather}
\label{qv}\sum_{i=1}^n \E\Prb*{ |\chi_i^n|^2 \given \mathcal{F}_{\frac{i-1}{n}} }\convp[n\to\infty] 0,\\
\label{ort}
\text{and }
\sum_{i=1}^n \E\Prb*{ \chi_i^n \Delta_i^n B  \given \mathcal{F}_{\frac{i-1}{n}} }\convp[n\to\infty] 0.
\end{gather}
Then $Z^n_1$ converges to $0$ in probability as $n\to\infty$.
\end{theorem}

\begin{remark}
In \cite{j2} some kind of uniform integrability is assumed in the limit, whereas here we do not ask explicitly for such a condition. The reason is that the uniform integrability assumption is implied by the fact that the limit in \eqref{qv} is $0$.

It is also required that  
$\sum_{i=1}^n \E\Prb*{ \chi_i^n \Delta_i^n N  \given \mathcal{F}_{\frac{i-1}{n}} }$ converges
 to $0$ for any 
bounded martingale $N$ orthogonal to $B$ on $(\Omega,\mathcal{F},\mathbb{F},\mathbb{P})$. 
As we have considered the Brownian motion with its natural (completed) filtration,
this set is reduced to the constant ones. %TODO  reference
\end{remark}

%%%%%%%%%%%%%%%%%%%%%%%%%%%%%%%%%%%%%%%%%%%%%%%%%%%%%%%%%%%%%%%%%%%%%%
\subsection{Scaled quadratic variation of Brownian local time}

Let $(\b_t)_{(t\in[0,1])}$ be a BM and $L(\b)$ its local time at $0$. Let us recall 
the \emph{diffusive scaling property} $(\beta_{ct},L_{t}(\beta))_{t>0}\eqlaw (\sqrt{c}\beta,\sqrt{c} L(\beta))$
for any $c>0$ (see \textit{e.g.} \cite[Exercise 2.11, p.~244]{ry}).

Let $\mathscr{H}=(\mathcal{H}_t)_{t\in[0,1]}$ be the natural (completed) filtration of $\beta$.

For $i=1,\dotsc,n$, we write $\mathcal{H}_{i,n}=\mathcal{H}_{i/n}$. 

\begin{lemma}
\label{qvb}
Let $L(\b)$ be the Brownian local time at $0$. The following convergence holds:
\[
\sqrt{n} [L(\beta)]_1^n
    =
\sqrt{n} \sum_{i=1}^n (L_{i,n}(\b)-L_{i-1,n}(\b))^2
    \convp[n\to\infty]
\frac{4\sqrt{2} }{3\sqrt{\pi}}L_1(\b).
\]
\end{lemma}

We split the proof of this result in the next tree lemmas.
We start with the explicit computations on the moments of the Brownian local time.

\begin{lemma}
    \label{lem:moments}
For $p\geq 1$, we set $\phi_p(\a):=\E\Prb*{L_1(\b)^p \given\b_0=\a}$.
We have
\begin{equation}
    \label{intphi}
\leb(\phi_p)=\frac{2}{p+1} \E|\mathcal{N}|^{p+1},
\end{equation}
where $\mathcal{N}$ denotes a standard Gaussian random variable.
Besides, the following tail estimates hold for $p=2,4$:
\begin{equation}
    \label{tailphi}
\phi_2(\a)\leq
\frac{e^{-\a^2/2}}{\a\sqrt{2\pi}}
\text{ and }
\phi_4(\a)\leq
\frac{16\sqrt{2}}{\sqrt{\pi} } \frac{e^{-\a^2/4}}{\a}.
\end{equation}
\end{lemma}

\begin{remark}\label{appl41}
These functions $\phi_p$ will be useful when applying Lemma~\ref{convskew} and Remark~\ref{j41}, taking, for fixed $p$, $f=\phi_p$. 
Inequalities~\eqref{tailphi} imply that the integrability condition for $f$ is satisfied and the theorem can be applied.
\end{remark}

\begin{proof}
Formula (6) in \cite{ta} gives the following expression for the moments of the Brownian local time
\[
\phi_p(\a)=2p \int_0^\infty x^{p-1} \PR(\mathcal{N}\geq |\a|+x)\vd x.
\]
To apply Remark~\ref{j41} we need to compute the following integral
\begin{equation*}
\leb(\phi_p) = \int_{-\infty}^\infty \phi_p(\a) \vd\a
=
2\int_0^\infty \phi_p(\a) \vd \a
=2
\int_0^\infty
2p \int_0^\infty x^{p-1} \int_0^\infty \frac{e^{-t^2/2}}{\sqrt{2\pi}}1_{[t\geq \a+x]}\vd t\vd x\vd\a.
\end{equation*}
Changing the order of integration by Fubini-Tonelli's theorem,
\[
\leb(\phi_p)=4p\int_0^\infty
\frac{e^{-t^2/2}}{\sqrt{2\pi}} \vd t
\int_0^\infty x^{p-1}\vd x\int_0^\infty 
1_{[t\geq \a+x]}\vd\a=\frac{2}{p+1} \E|\mathcal{N}|^{p+1},
\]
so \eqref{intphi} is proved.
We now use the following bound for Gaussian tails: $\int_x^\infty \frac{e^{-t^2/2}}{\sqrt{2\pi}}\vd t\leq \frac{e^{-x^2/2}}{x\sqrt{2\pi}}$. 
We apply it twice and find the upper bound for $p=2$:
\[
\phi_2(\a)=
4\int_0^\infty x \int_0^\infty \frac{e^{-t^2/2}}{\sqrt{2\pi}}1_{[t\geq \a+x]}\vd t\vd x
\leq
\frac{e^{-\a^2/2}}{\a\sqrt{2\pi}}.
\]
For $p=4$ we apply the same inequality:
\begin{multline*}
\phi_4(\a)
 =
8 \int_0^\infty x^{3} \int_0^\infty \frac{e^{-t^2/2}}{\sqrt{2\pi}}1_{[t\geq \a+x]}\vd t\vd x\\
\leq
8 \int_0^\infty x^3 \frac{e^{-(\a+x)^2/2}}{(\a+x)\sqrt{2\pi}}\vd x\leq
8 \int_0^\infty x^2 \frac{e^{-(\a+x)^2/2}}{\sqrt{2\pi}}\vd x.
\end{multline*}
Now, since $x e^{-x}\leq e^{-1}$ for all $x\geq 0$,
\[
\phi_4(\a)
\leq
16 \int_0^\infty \frac{e^{-(\a+x)^2/4}}{\sqrt{2\pi}}\vd x \leq
32 \frac{e^{-\a^2/4}}{\a\sqrt{2\pi}}.
\]
Hence the result.
\end{proof}

We consider now the quadratic sum in Lemma \ref{qvb}, and write 
\[
\begin{split}
\sqrt{n} \sum_{i=1}^n (L_{i,n}(\b)&-L_{i-1,n}(\b))^2
=
\sqrt{n} \sum_{i=1}^n \E[(L_{i,n}(\b)-L_{i-1,n}(\b))^2| \mathcal{H}_{i-1,n}]\\
&+
\sqrt{n} \sum_{i=1}^n \left((L_{i,n}(\b)-L_{i-1,n}(\b))^2
-
\E[(L_{i,n}(\b)-L_{i-1,n}(\b))^2| \mathcal{H}_{i-1,n}]\right).
\end{split}
\]
In the next two lemmas we prove the convergence of the two summands. Lemma \ref{qvb} follows directly.

\begin{lemma}
Let $L(\b)$ be the Brownian local time at $0$. The following convergence holds:
\[
\sqrt{n} \sum_{i=1}^n \E[(L_{i,n}(\b)-L_{i-1,n}(\b))^2| \mathcal{H}_{i-1,n}]
\convp[n\to\infty] \frac{4\sqrt{2}}{3\sqrt{\pi}}
 L_1(\b).
\]
\end{lemma}

\begin{proof}
The diffusive scaling property of $(\b,L(\b))$ implies that for any $p\geq 1$, 
\begin{equation}
    \label{eqmom}
\E\Prb*{(L_{i,n}(\b)-L_{i-1,n}(\b))^{p}\given\mathcal{H}_{i-1,n}}=
\frac{1}{n^{p/2}} \E\Prb*{L_1(\b)^{p}\given \b_0=\sqrt{n}\b_{i-1,n}}.
\end{equation}
Setting $p=2$ and since $\leb(\phi_2)=\frac{2}{3} \E |\mathcal{N}|^3$ from \eqref{intphi},  
Remark~\ref{j41} below Lemma~\ref{convskew} implies that 
\begin{equation*}
\sqrt{n} \sum_{i=1}^n \E[(L_{i,n}(\b)-L_{i-1,n}(\b))^2| \mathcal{H}_{i-1,n}]
=
 \sum_{i=1}^n \frac{1}{\sqrt{n}} \phi_2( \b_{i-1,n} \sqrt{n})
 \convp[n\to\infty] \frac{4\sqrt{2}}{3\sqrt{\pi}}L_1(\b).
\end{equation*}
Hence the result.
\end{proof}

We consider now the martingale part.

\begin{lemma}
With 
\begin{equation*}
H_{i,n} := (L_{i,n}(\b)-L_{i-1,n}(\b))^2- \E[(L_{i,n}(\b)-L_{i-1,n}(\b))^2| \mathcal{H}_{i-1,n}],
\end{equation*}
it holds that $\sqrt{n}\sum_{i=1}^n H_{i,n}\convp[n\to\infty] 0$.
\end{lemma}

\begin{proof} 
    The statement is proved using Theorem \ref{jm} by setting $\chi_i^n:=\sqrt{n}H_{i,n}$.

\begin{itemize}[leftmargin=0em,itemindent=1em]
\item We prove fist \eqref{qv}. From \eqref{eqmom} with $p=4$,
\begin{equation*}
\E\Prb*{H_{i,n}^2\given\mathcal{H}_{i-1,n}} 
\leq
 \E\Prb*{(L_{i,n}(\b)-L_{i-1,n}(\b))^4\given \mathcal{H}_{i-1,n}}
=
\frac{1}{n^2} \phi_4( \b_{i-1,n} \sqrt{n}) .
\end{equation*}
With Remark \ref{j41} below Lemma~\ref{convskew},
$n^{-1/2} \sum_{i=1}^n \phi_4( \b_{i-1,n} \sqrt{n})$
converges in probability to $\leb(\phi_4)L_1(\beta)$ 
because of \eqref{tailphi}. 
Thus,  
\begin{equation*}
n \sum_{i=1}^n \E\Prb*{H_{i,n}^2\given\mathcal{H}_{i-1}}
\leq \frac{1}{\sqrt{n}}\left(
\frac{1}{\sqrt{n}}
\sum_{i=1}^n \phi_4( \b_{i-1,n} \sqrt{n}) 
\right)
\convp[n\to\infty]0.
\end{equation*}

\item
We take $B=\b$ in \eqref{ort}. We have
\begin{multline*}
\sum_{i=1}^n \E \Prb*{H_{i,n} (\b_{i,n}-\b_{i-1,n})\given\mathcal{H}_{i-1,n}}
=
\sum_{i=1}^n \E \Prb*{ (L_{i,n}(\b)-L_{i-1,n}(\b))^2 (\b_{i,n}-\b_{i-1,n})\given \mathcal{H}_{i-1,n}}\\
- \E\Prb*{(L_{i,n}(\b)-L_{i-1,n}(\b))^2\given \mathcal{H}_{i-1,n}}
\E\Prb*{(\b_{i,n}-\b_{i-1,n})\given \mathcal{H}_{i-1,n}}.
\end{multline*}
Since $\E\Prb*{(\b_{i,n}-\b_{i-1,n})\given \mathcal{H}_{i-1,n}}=0$,
we only estimate the first summand:
\begin{multline*}
    \left| \E\Prb*{(L_{i,n}(\b)-L_{i-1,n}(\b))^2 (\b_{i,n}-\b_{i-1,n}) \given \mathcal{H}_{i-1,n}} \right|\\
\leq
\E\Prb*{(L_{i,n}(\b)-L_{i-1,n}(\b))^4 \given \mathcal{H}_{i-1,n}}^{1/2}
\E\Prb*{ (\b_{i,n}-\b_{i-1,n})^2 \given \mathcal{H}_{i-1,n}}^{1/2}.
\end{multline*}
We estimate the two factors:
\[
\E\Prb*{ (\b_{i,n}-\b_{i-1,n})^2 \given \mathcal{H}_{i-1,n}}^{1/2}
\leq \frac{1}{\sqrt{n}}
\]
and from \eqref{eqmom} with $p=4$,
\[
\E\Prb*{(L_{i,n}(\b)-L_{i-1,n}(\b))^4 \given \mathcal{H}_{i-1,n}}^{1/2}
\leq 
\frac{1}{n}\phi_4(\b_{i-1,n}\sqrt{n})^{1/2}.
\]
Therefore, from Remark \ref{j41}, that can be applied because of \eqref{tailphi},
\begin{equation*}
\sqrt{n}\left|\sum_{i=1}^n \E \Prb*{ H_{n,i} (\b_{i,n}-\b_{i-1,n})\given \mathcal{H}_{i-1,n} }\right|
\leq 
\frac{1}{\sqrt{n}}\left(\frac{1}{\sqrt{n}} \sum_{i=1}^n \phi_4(\b_{i-1,n}\sqrt{n})^{1/2}
\right)
\convp[n\to\infty] 0.
\end{equation*}
\end{itemize}
The proof is then complete.
\end{proof}
%%

%%%%%%%%%%%%%%%%%%%%%%%%%%%%%%%%%%%%%%%%%%%%%%%%%%%%%%%%%%%%%%%%%%%%%%
\subsection{Scaled quadratic covariation of skew Brownian motion and its local time}

We now give some results on the scaled quadratic covariation between the SBM and its local time.
For the Brownian motion $W$ with the filtration $\mathscr{G}=(\mathcal{G}_t)_{t\geq0}$ of Section~\ref{sectionOBM}, 
we consider~$X$ the strong solution to $X_t=x+W_t+\theta L_t(X)$ for $\theta\in[-1,1]$ and $L(X)$ its local 
time (apart from the results in \cite{legall}, strong existence for the SBM has been proved first in \cite{harrison-shepp_1981}).

\begin{lemma}
\label{qcs}
For $X$ and $L(X)$ as above, the following convergence holds:
\begin{gather}
\label{eqqcs} \sqrt{n}[X,L(X)]_1^n=\sqrt{n} \sum_{i=1}^n (X_{i,n}-X_{i-1,n})(L_{i,n}(X)-L_{i-1,n}(X))\convp[n\to\infty]
0,\\
\label{eqmo} \sqrt{n}[|X|,L(X)]_1^n=\sqrt{n} \sum_{i=1}^n (|X_{i,n}|-|X_{i-1,n}|)(L_{i,n}(X)-L_{i-1,n}(X))\convp[n\to\infty]
0,\\
\label{eqpo} \sqrt{n}[X^+,L(X)]_1^n=\sqrt{n} \sum_{i=1}^n (X_{i,n}^+-X_{i-1,n}^+)(L_{i,n}(X)-L_{i-1,n}(X))\convp[n\to\infty]
0.
\end{gather}
\end{lemma}

We set
\begin{equation}
    \label{eq:Zin}
Z_{i,n}:=(X_{i,n}-X_{i-1,n})(L_{i,n}(X)-L_{i-1,n}(X))
\end{equation}
and write
\[
\begin{split}
\sqrt{n} \sum_{i=1}^n Z_{i,n}
=
\sqrt{n} \sum_{i=1}^n \E[Z_{i,n}| \G_{i-1,n}] +
\sqrt{n} \sum_{i=1}^n \big(Z_{i,n}
-\E[Z_{i,n}| \G_{i-1,n}]\big).
\end{split}
\]

We prove \eqref{eqqcs} in the next two lemmas. Once \eqref{eqqcs} is proved, 
\eqref{eqmo} follows since $|X|$ is a SBM with parameter $\theta=1$,
while \eqref{eqpo} follows from a combination of \eqref{eqqcs} and \eqref{eqmo}
since $X^+=\frac{|X|+X}{2}$.

\begin{lemma}
\label{l3}
With $Z_{i,n}$ defined in \eqref{eq:Zin}, the following convergence holds:
\begin{equation}
    \label{eqcondqcs}
\sqrt{n} \sum_{i=1}^n \E \Prb*{ Z_{i,n} \given \G_{i-1,n}  }
\convp[n\to\infty] 0.
\end{equation}
\end{lemma}

\begin{proof}
We express first $\E(X_t-x)^2$ as a function of $x$
using the law of the SBM.
The density transition function of the SBM is \cite{walsh,le}
\begin{equation*}
p_\th(t,x,y):=p(t,x-y)+\sgn(y) \th  p(t,|x|+|y|)
\end{equation*}
where $p(t,x)=(2 \pi t)^{-1/2}e^{-x^2/(2t)}$, the Gaussian density. 
Therefore 
\begin{equation*}
\E (X_t-x)^2 =\E B_t^2+\th t \psi(x/\sqrt{t})
\text{ with } 
\psi(x):=\int_{-\infty}^\infty (x-y)^2 \sgn(y)\frac{e^{-(|x|+|y|)^2/2}}{\sqrt{2\pi}}\vd y.
\end{equation*}
We compute $\psi$ for $x>0$:
\begin{multline*}
\psi (x)=
-\int_{-\infty}^0 
(x-y)^2\frac{e^{-(x-y)^2/2}}{\sqrt{2\pi}}\vd y
+\int_0^\infty 
(x-y)^2\frac{e^{-(x+y)^2/2}}{\sqrt{2\pi}}\vd y
\\
=-\int_x^{\infty} 
z^2\frac{e^{-z^2/2}}{\sqrt{2\pi}}\vd z
+\int_0^\infty 
(x-y)^2\frac{e^{-(x+y)^2/2}}{\sqrt{2\pi}}\vd y
\end{multline*}
and
\begin{multline*}
\int_0^\infty (x-y)^2\frac{e^{-(x+y)^2/2}}{\sqrt{2\pi}}\vd y 
=
\int_0^\infty (x+y)^2\frac{e^{-(x+y)^2/2}}{\sqrt{2\pi}}\vd y
-4x\int_0^\infty y\frac{e^{-(x+y)^2/2}}{\sqrt{2\pi}}\vd y\\
=
\int_x^\infty z^2\frac{e^{-z^2/2}}{\sqrt{2\pi}}\vd z
-4x\int_0^\infty (y+x)\frac{e^{-(x+y)^2/2}}{\sqrt{2\pi}}\vd y
+4x^2\int_0^\infty \frac{e^{-(x+y)^2/2}}{\sqrt{2\pi}}\vd y \\
=
\int_x^\infty z^2\frac{e^{-z^2/2}}{\sqrt{2\pi}}\vd z
-4x\int_x^\infty z\frac{e^{-z^2/2}}{\sqrt{2\pi}}\vd z
+4x^2\int_x^\infty \frac{e^{-z^2/2}}{\sqrt{2\pi}}\vd z.
\end{multline*}
So for $x>0$
\begin{equation*}
\psi(x)=-4x\int_x^\infty z\frac{e^{-z^2/2}}{\sqrt{2\pi}}\vd z
+4x^2\int_x^\infty \frac{e^{-z^2/2}}{\sqrt{2\pi}}\vd z=
4x(x(1-\Phi(x))-p(1,x))
\end{equation*}
and
\begin{equation*}
\int_0^\infty
\psi(x)\vd x=
2\left(\frac{\E[|\mathcal{N}|^3]}{3} -\E[|\mathcal{N}|]\right)
=
-\frac{2\sqrt{2}}{3\sqrt{\pi}} .
\end{equation*}
With the change of variable $y\rightarrow -y$, we see that $\psi$ is an odd function. 
Thus, $\int_{-\infty}^\infty\psi(x)\vd x=0$.

Recall now \eqref{sbm}.
Writing $(X_t-x )- \th L_t(X)=B_t$, 
\begin{equation*}
(X_t-x)^2+\th^2 L_t(X)^2-2 \th(X_t-x)  L_t(X) =B_t^2.
\end{equation*}
Recall that $(|X|,L(X))\eqlaw(|\b|,L(\b))$, where $\b$ is a BM. Moreover, $\phi_2$ 
    defined in Lemma~\ref{lem:moments} is symmetric. So
\begin{equation*}
\E L_t(X)^2=
\E L_t(\b)^2=
t\phi_2(\b_0/\sqrt{t})=
t\phi_2(X_0/\sqrt{t}).
\end{equation*}
Therefore
\begin{equation*}
\E(X_t-x)  L_t(X) =\frac{\th }{2}\E L_t(X)^2 + 
\frac{1}{2\th}(\E (X_t-x)^2-\E B_t^2)
= \frac{t\th }{2}\phi_2(x/\sqrt{t}) + \frac{t}{2} \psi(x/\sqrt{t})
\end{equation*}
and
\begin{equation*}
\E \Prb*{ (X_{i,n}-X_{i-1,n})(L_{i,n}(X)-L_{i-1,n}(X))\given \G_{i-1,n}  }
=
\frac{1}{2n} (\th\phi_2(X_{i-1,n} \sqrt{n})
+ \psi(X_{i-1,n} \sqrt{n}) ).
\end{equation*}
Since $\phi_2$ is symmetric and applying \eqref{la1},
\begin{equation*}
\bar{\l}_\th(\phi_2)=\leb(\phi_2)=2\frac{\E[|\mathcal{N}|^3]}{3}.
\end{equation*}
Since $\psi$ is anti-symmetric and \eqref{la1}
\begin{equation*}
\bar{\l}_\th(\psi)=(1+\th)\leb(\psi^+)+(1-\th)\leb(\psi^-)=2\th \leb(\psi^+)=
4\th
\left(\frac{\E[|\mathcal{N}|^3]}{3} -\E[|\mathcal{N}|]\right)
\end{equation*}
so
\begin{equation*}
\bar{\l}_\th\left(\frac{\th\phi_2}{2}+\frac{\psi}{2}\right)=
\th
\big(\E[|\mathcal{N}|^3] -2\E[|\mathcal{N}|]\big)
=0.
\end{equation*}
It is straightforward to check that $\int |x|^k (\frac{\th}{2}\phi_2(x)+\frac{1}{2}\psi(x))\vd x<\infty$
for $k=0,1,2$. With Lemma~\ref{convskew}, this proves~\eqref{eqcondqcs}.
\end{proof}

\begin{lemma}
\label{l4}
With $Z_{i,n}$ defined by \eqref{eq:Zin}, the following convergence holds:
\begin{equation*}
\sqrt{n} \sum_{i=1}^n
\left(  Z_{i,n}- \E \Prb*{ Z_{i,n} \given \G_{i-1,n} } \right)
\convp[n\to\infty] 0.
\end{equation*}
\end{lemma}

\begin{proof}
We mean to apply Theorem \ref{jm}.
We first prove \eqref{qv}:
\begin{multline*}
\E \Prb*{ (Z_{i,n}- \E \Prb*{ Z_{i,n} \given\G_{i-1,n} })^2 \given \G_{i-1,n} }\\
\leq
\E\Prb*{ (X_{i,n}-X_{i-1,n})^2(L_{i,n}(X)-L_{i-1,n}(X))^2 \given \G_{i-1,n}}\\
\leq
\E\Prb*{ (X_{i,n}-X_{i-1,n})^4 \given \G_{i-1,n}}^{1/2}
\E\Prb*{ (L_{i,n}(X)-L_{i-1,n}(X) )^4 \given \G_{i-1,n}}^{1/2}
\end{multline*}
and we upper bound the two factors. We know
$\E\big[ (X_{i,n}-X_{i-1,n})^4 |
\G_{i-1,n}\big]^{1/2}
\leq
 \frac{C}{n}$.
Recall again that $(|X|,L(X) )\eqlaw(|\b|,L(\b))$, where $\b$ is a BM, and that $\phi_4$ is symmetric.
From \eqref{eqmom},
\[
\E\Prb*{(L_{i,n}(X)-L_{i-1,n}(X))^4\given \G_{i-1,n}}^{1/2} \leq \frac{1}{n}\phi_4(X_{i-1,n}\sqrt{n})^{1/2}.
\]
 Because of \eqref{tailphi}, we apply Lemma \ref{convskew} so that 
\begin{equation*}
n\E \Prb*{ (Z_{i,n}-
\E \Prb*{ Z_{i,n} \given\G_{i-1,n} })^2 \given\G_{i-1,n} }
\leq
 \frac{C}{\sqrt{n}}  \sum_{i=1}^n \frac{1}{\sqrt{n}} \phi_4( X_{i-1} \sqrt{n})^{1/2}\convp[n\to\infty] 0.
\end{equation*}

Since \eqref{obm} has a strong solution,
we take $B=W$, the BM driving \eqref{obm}, in \eqref{ort}. 
Since $\E\Prb*{(W_{i,n}-W_{i-1,n})\given \G_{i-1,n}}=0$, 
\begin{multline*}
\sqrt{n} \sum_{i=1}^n \E\Prb*{
(Z_{i,n}-
\E \Prb*{ Z_{i,n} \given\G_{i-1,n} })(W_{i,n}-W_{i-1,n})
\given \G_{i-1,n}}\\
=
\sqrt{n} \sum_{i=1}^n \E\Prb*{
Z_{i,n}(W_{i,n}-W_{i-1,n})\given \G_{i-1,n}}.
\end{multline*}
Since $Z_{i,n}$ defined in \eqref{eq:Zin} is the product of the increments of $X$ with the ones of the local time,
\begin{multline*}
    \left|\E\Prb*{Z_{i,n}(W_{i,n}-W_{i-1,n})\given \G_{i-1,n}}\right|\leq
\E\Prb*{(X_{i,n}-X_{i-1,n})^4 \given \G_{i-1,n}}^{1/4}\\
\times \E\Prb*{(W_{i,n}-W_{i-1,n})^4 \given \G_{i-1,n}]^{1/4}}
\E\Prb*{(L_{i,n}(X)-L_{i-1,n}(X))^2 \given \G_{i-1,n}}^{1/2}.
\end{multline*}
Now,
\begin{equation*}
\E\Prb*{(W_{i,n}-W_{i-1,n})^4\given \G_{i-1,n}}^{1/4}\leq \frac{1}{\sqrt{n}}\text{ and }
\E\Prb*{(X_{i,n}-X_{i-1,n})^4 \given\G_{i-1,n}}^{1/4}\leq \frac{1}{\sqrt{n}}.
\end{equation*}
From \eqref{eqmom},
$(|X|,L(X) )\eqlaw(|\b|,L(\b))$, with $\b$  BM, and $\phi_2$  symmetric
\[
\E\Prb*{(L_{i,n}(X)-L_{i-1,n}(X))^2 \given \G_{i-1,n}}^{1/2}
\leq 
\frac{1}{\sqrt{n}}\phi_2(X_{i-1,n}\sqrt{n})^{1/2}.
\]
Therefore
\begin{equation*}
    \big|\E\Prb*{Z_{i,n}(W_{i,n}-W_{i-1,n})\given \G_{i-1,n}}\big|\leq
\frac{1}{n\sqrt{n}}\phi_2(X_{i-1,n}\sqrt{n})^{1/2}.
\end{equation*}
From Lemma \ref{convskew},
\begin{multline*}
\sqrt{n} \sum_{i=1}^n 
\big|
\E\Prb*{(Z_{i,n}- \E \Prb*{ Z_{i,n} \given\G_{i-1,n} })(W_{i,n}-W_{i-1,n})\given \G_{i-1,n}}
\big|
\\
\leq
\frac{1}{\sqrt{n}}\left( \frac{1}{\sqrt{n}}\sum_{i=1}^n \phi_2(X_{i-1,n}\sqrt{n})^{1/2}\right)\convp[n\to\infty] 0.
\end{multline*}
Hence the result.
\end{proof}

%%%%%%%%%%%%%%%%%%%%%%%%%%%%%%%%%%%%%%%%%%%%%%%%%%%%%%%%%%%%%%%%%%%%%%

\subsection{Approximation of occupation time}\label{secot}

In this section we extend the result in \cite{no}, which is proved for diffusions with smooth coefficients, to the OBM.
We consider approximating the occupation time of $[0,\infty)$ up to time $1$:
\[
    Q^+_1=\leb(s\in[0,1]:Y_s\geq 0) = \int_0^1 1_{\{Y_s\geq 0\}} \vd s.
\]
As previously, we suppose that we know the values $Y_{i,n}$ of $Y$ on a grid of time lag $1/n$.

\begin{theorem}\label{ot} Let $Y$ be given in \eqref{obm} and $\bar{Q}_1^n(Y,+)$ be given by \eqref{eot}.
The following convergence holds:
\[
\sqrt{n} \left(\bar{Q}_1^n(Y,+) -
\int_0^t 1_{\{Y_s\geq 0\}}\right)
 \convp[n\to\infty] 0.
\]
\end{theorem}

For $i=1,\dotsc ,n$, we consider 
\begin{equation}\label{setJU}
\begin{split}
J_{i,n}&= \left(\frac{1}{n} 1_{\{Y_{i-1,n}\geq 0\}}
-
 \int_{\frac{i-1}{n}}^{\frac{i}{n}} 1_{\{Y_s\geq 0\}} \vd s \right)=
\sgn(Y_{i-1,n}) 
\int_{\frac{i-1}{n}}^{\frac{i}{n}} 1_{\{Y_{i-1,n}
 Y_s<0\}} \vd s, \\
U_{i,n}&=J_{i,n}-\E\Prb*{J_{i,n}\given \G_{i-1,n}} .
 \end{split}
 \end{equation}
The $U_{i,n}$ are martingale increments.
We write
\[
\sqrt{n}\left(\bar{Q}_1^n(Y,+) -
\int_0^1 1_{\{Y_s\geq 0\}}\right)
=
\sqrt{n}\sum_{1=1}^n \E\Prb*{ J_{i,n}  \given \G_{i-1,n} }
+
\sqrt{n}\sum_{1=1}^n U_{i,n}.
\]
In the following lemmas we prove the convergence of the two summands. 

\begin{remark}
In \cite{no} it is proved that the estimator times $n^{3/4}$ is tight, so the
speed of convergence proved there, holding only for smooth coefficients, is
faster than the speed proved here. We are actually able to prove that $n^{3/4}$
is the speed of convergence for the martingale part $\sum_{1=1}^n U_{i,n}$ also
for the OBM (and other diffusions with discontinuous coefficients), but not for
the drift part $\sum_{1=1}^n \E\Prb*{ J_{i,n}  \given\G_{i-1,n} }$. We would need the
central limit theorem giving the speed of convergence corresponding to the law
of large numbers proved in Lemma \eqref{convY}, but this looks quite hard to
get in the case of discontinuous diffusion coefficients. Anyways, for our goal
of estimating the parameters of the OBM, the fact that our estimator multiplied
with the ``diffusive scaling'' $\sqrt{n}$ converges to $0$ is enough. Actually,
this result depends on a compensation between two terms in Lemma \ref{ctot}
which holds for this particular diffusion but for which we do not have results
holding for a wider class of SDEs with discontinuous coefficients.
\end{remark}

\begin{lemma}
\label{ctot} With $J_{i,n}$ defined by \eqref{setJU}, the following convergence holds:
\[
\sqrt{n}\sum_{1=1}^n \E\Prb*{ J_{i,n}  \given \G_{i-1,n} }
\convp[n\to\infty] 0.
\]
\end{lemma}

\begin{proof}
From \eqref{setJU},
\begin{multline*}
\sqrt{n} \E\Prb*{ \int_{\frac{i-1}{n}}^{\frac{i}{n}}  \sgn(Y_{i-1,n}) 1_{\{Y_{i-1,n}Y_s<0\}} \vd s \given \G_{i-1,n}  }\\
=
\sqrt{n}  \int_{\frac{i-1}{n}}^{\frac{i}{n}}  
\sgn(Y_{i-1,n})  \E \Prb*{1_{\{Y_{i-1,n}Y_s<0\}} \given \G_{i-1,n}  } \vd s
\\
=
 \sqrt{n} \int_{\frac{i-1}{n}}^{\frac{i}{n}}  
 \sgn(Y_{i-1,n}) \PR \Prb*{Y_{i-1,n}Y_s<0 \given \G_{i-1,n}  } \vd s.
\end{multline*}
Using the Markov property and \eqref{propone} we can compute this quantity.
When $Y_0>0$
\begin{multline*}
 \sqrt{n} \int_{0}^{\frac{1}{n}}  
 \sgn(Y_{0}) \PR \left(Y_{0}Y_s<0\right) \vd s
=
 \int_0^{\frac{1}{n}}  
 \frac{2\s_-}{\s_-+\s_+}\Phi(Y_0/(\s_-\sqrt{s}))
  \vd s\\
  =
 \frac{\sqrt{n} 2\s_-}{\s_-+\s_+} \int_0^{1}  
\Phi(Y_0 \sqrt{n}/(\s_-\sqrt{t})) \frac{\vd t}{n} 
=
\frac{1}{\sqrt{n}}  f(Y_0 \sqrt{n}),
\end{multline*}
where
\[
  f(x) := \frac{ 2\s_+}{\s_-+\s_+} \int_0^{1}  
\Phi(-x/(\s_+\sqrt{t})) \vd t\text{ for }x>0. 
\]
Now, for $Y_0<0$, we find
\[
 \sqrt{n} \int_{0}^{\frac{1}{n}}  
 \sgn(Y_{0}) \PR \left(Y_{0}Y_s<0\right) \vd s
=\frac{1}{\sqrt{n}}  f(Y_0 \sqrt{n}),
\]
where
\[
  f(x) := \frac{ -2\s_-}{\s_-+\s_+} \int_0^{1}  
\Phi(x/(\s_-\sqrt{t})) \vd t \text{ for }x<0.
\]
We can compute
\[
\int_0^\infty f(x) \vd x=
 \int_0^\infty
 \int_0^1  
 \frac{2\s_+}{\s_-+\s_+}\Phi(-x/(\s_+\sqrt{s}))
  \vd s \vd x
  =
\frac{2\s_+^2}{\s_-+\s_+} \int_0^\infty
 \int_0^1  
 \Phi(-x/\sqrt{s})
  \vd s \vd x
 \]
and
\[
 \int_{-\infty}^0 f(x) \vd x=
  \frac{-2\s_-^2}{\s_-+\s_+}
 \int_{-\infty}^0
 \int_0^1  
\Phi(x/\sqrt{s})
  \vd s \vd x.
\]
Therefore computing the coefficient in \eqref{la} we find
\[
\l_\s(f)=\frac{2}{\s_+^2} \int f^+
+
\frac{2}{\s_-^2} \int f^-
=0.
\]
Moreover, 
\[
|  f(x)|\leq C \int_0^{1}  
\Phi(C x/\sqrt{t})) \vd t 
\leq C e^{-C x^2}
\]
for some constant $C$, so $\int |x|^k f(x)\vd x<\infty$ for all $k\geq 0$. 
Applying Lemma~\ref{convY}, we prove the statement.
\end{proof}

\begin{lemma}
With $U_{i,n}$ defined by \eqref{setJU}, the following convergence holds:
\[
\sqrt{n} \sum_{i=1}^n  U_{i,n} \convp[n\to\infty]0.
\]
\end{lemma}

\begin{proof}
We consider \eqref{qv}. We have
\begin{equation*}
\E\Prb*{ U_{i,n}^2\given\G_{i-1,n} }= \E\Prb*{ J_{i,n}^2\given\G_{i-1,n} }
-\E\Prb*{ J_{i,n}\given \G_{i-1,n} }^2\leq \E\Prb*{ J_{i,n}^2\given\G_{i-1,n} }.
\end{equation*}
From \eqref{setJU},
\begin{multline*}
\E\Prb*{ J_{i,n}^2\given \G_{i-1,n} } 
\leq \frac{1}{n} \E\Prb*{ \int_{\frac{i-1}{n}}^{\frac{i}{n}}   1_{\{(Y_{i-1,n}Y_s)<0\}} \vd s \given \G_{i-1,n}  }\\
=\frac{1}{n}
 \int_{\frac{i-1}{n}}^{\frac{i}{n}}  
 \PR \Prb*{Y_{i-1,n}Y_s<0 \given \G_{i-1,n}  } \vd s 
=\frac{1}{n^2}
|f(Y_{i-1,n} \sqrt{n}) |.
 \end{multline*}
Therefore, in \eqref{la},
\[
\l_\s(|f|)=\frac{2}{\s_+^2} \int f^+
-
\frac{2}{\s_-^2} \int f^-
=\frac{4}{\s_-+\s_+}
 \int_{-\infty}^\infty
 \int_0^1  
\Phi(x/\sqrt{s})
  \vd s \vd x.
\]
Lemma \ref{convY} implies
\[
n \sum_{i=1}^n \E\Prb*{ J_{i,n}^2\given \G_{i-1,n} } =  \sum_{i=1}^n \frac{1}{n}
f(Y_{i-1,n} \sqrt{n}) 
\convp[n\to\infty] 0.
\]

Now we check \eqref{ort}. Let $W$ be the BM driving \eqref{obm}.
\begin{multline*}
|\E\Prb*{ U_{i,n}(W_{i,n}-W_{i-1,n})\given  \G_{i-1,n} }|
\leq
\E\Prb*{ U_{i,n}^2\given  \G_{i-1,n} }^{1/2}
\E\Prb*{ (W_{i,n}-W_{i-1,n})^2\given  \G_{i-1,n} }^{1/2}\\
\leq
\frac{1}{n}\sqrt{|f(Y_0\sqrt{n})|}\frac{1}{\sqrt{n}}
\end{multline*}
and
\begin{equation*}
\sqrt{n}|\sum_{i=1}^{[nt]}\E[U_{i,n}(W_{i,n}-W_{i-1,n})| \G_{i-1,n}]|
\leq
\frac{1}{\sqrt{n}}\sum_{i=1}^{[nt]}\frac{1}{\sqrt{n}}\sqrt{|f(Y_0\sqrt{n})|}\convp[n\to\infty]0
\end{equation*}
from another application of Lemma \ref{convY}. The condition of integrability is satisfied because 
\[
\sqrt{|  f(x)|} 
\leq C e^{-C x^2},
\]
for some constant $C$, so $\int |x|^k \sqrt{|  f(x)|} \vd x<\infty$ for all $k\geq 0$.
\end{proof}

\subsection{Proof of the main results}
In this section we use all the results proved so far to prove Theorems \ref{maintheorem} and \ref{maintheorem2}.

\begin{proof}[Proof of Theorem \ref{maintheorem}]
We set
\begin{equation*}
\xi_t= \s_+
\int_0^t \mathbf{1}(Y_s > 0)  \vd W_s\text{ and } 
\eta_t= \s_-
\int_0^t \mathbf{1}(Y_s < 0)  \vd W_s.
\end{equation*}
Itô-Tanaka's formula (see \cite{ry}) gives the following equation for the positive and negative part of~$Y$:
\begin{equation}\label{tanaka}
 Y_t^+= \xi_t + \frac{1}{2} L_t(Y)  , \text{ and }
 Y_t^-= -\eta_t + \frac{1}{2} L_t(Y). 
\end{equation}
Moreover, $\xi$ is a martingale with quadratic variation
\[
\langle \xi\rangle_t = \int_0^t \s(Y_s)^2 \mathbf{1}(Y_s>0) \vd s = \s_+^2\int_0^t  \mathbf{1}(Y_s> 0) \vd s.
\]
It is well known that the quadratic variation of a martingale can be approximated with the sum of squared increments over shrinking partitions. 
Thus, 
\begin{equation}\label{lln}
[\xi]^n_1\convp[n\to\infty] \langle \xi\rangle_1
= \int_0^1 \s_+^2 \mathbf{1}(Y_s> 0) \vd s
= \s_+^2 Q^+_1(Y).
\end{equation}
From \eqref{tanaka},
\begin{equation*}
    \label{est+}
[Y^+]^n_1 = [\xi]^n_1-\frac{[L(Y)]^n_1}{4} +[Y^+,L(Y)]^n_1.
\end{equation*}
The local time $L(Y)$ is of finite variation, $Y^+$ is continuous. Thus 
$[L(Y)]^n_1$ as well as $[Y^+,L(Y)]^n_1$ converge to $0$ almost surely. From \eqref{lln},
\begin{equation}\label{lln+}
[Y^+]^n_1
\convp[n\to\infty] \int_0^1 \s_+^2 \mathbf{1}(Y_s> 0) \vd s =\s_+^2 Q^+_1(Y).
\end{equation}
Recall the definition of $\bar{Q}^n_1(Y,+)$ in \eqref{eot}. Then
\begin{equation*}
    \label{cot}
\bar{Q}^n_1(Y,+)
\convas[n\to\infty]
\int_0^1 \mathbf{1}(Y_s\geq 0) \vd s =Q^+_1(Y).
\end{equation*}
From \eqref{lln+} and \eqref{cot},
$\hat{\s}^n_+ \convp \s_+$,
and similarly
$\hat{\s}^n_- \convp \s_-$.
Therefore, the vector $(\hat{\s}^n_+,\hat{\s}^n_-)$ converges in probability to $(\sigma_+,\sigma_-)$. 
The estimator  $(\hat{\s}^n_+,\hat{\s}^n_-)$ is then consistent. 

We consider now the rate of convergence. From \eqref{tanaka} applied to $Y^-$, we have as in \eqref{est+} 
that
\[
    \begin{pmatrix}
	[Y^+]^n_1
	\\
	[Y^-]^n_1
    \end{pmatrix}
	= 
	\begin{pmatrix} [\xi]^n_1\\ [\eta]^n_1 \end{pmatrix}
+
\begin{pmatrix} [Y^+,L(Y)]^n_1 \\ [Y^-,L(Y)]^n_1 \end{pmatrix}
-
\begin{pmatrix} 1\\ 1 \end{pmatrix} \frac{[L(Y)]^n_1}{4} .
\]
We consider separately the tree summands. From the central limit theorem for martingales (see for example \cite{jp}, (5.4.3) or Theorem 5.4.2), since $\mathbf{1}(Y_s > 0) \mathbf{1}(Y_s < 0)=0$,
\begin{equation*}
\sqrt{n}
\left(
\begin{pmatrix}
[\xi]_1^n & [\xi,\eta]_1^n \\ {}
[\eta,\xi]_1^n & [\eta]_1^n
\end{pmatrix}
-
\begin{pmatrix}
\langle\xi\rangle_1 & 0 \\
0 & \langle\eta\rangle_1
\end{pmatrix}
\right)
\convsl[n\to\infty] 
\sqrt{2}
\int_0^1  
\begin{pmatrix}
 \s^2_+ \mathbf{1}(Y_s > 0) & 0 \\
0 & \s^2_- \mathbf{1}(Y_s < 0)
\end{pmatrix}
\vd\bar{B}_s,
\end{equation*}
where $\bar{B}$ is a Brownian motion independent of the filtration of $W$. 
Therefore it is also independent of $L(Y)$.
Consider now the second summand. 
The OBM $Y$ is linked to a SBM $X$ solution to \eqref{sbm} through  $Y_t=X_t\s(X_t)$.
With \eqref{elt} and \eqref{eqpo} in Lemma~\ref{qcs}, 
\[
    \sqrt{n}
[Y^+,L(Y)]^n_1=
\sqrt{n}
\frac{2\s_+^2\s_-}{\s_++\s_-} [X^+,L(X)]^n_1
\convp[n\to\infty] 0.
\]
Clearly this also holds for $[Y^-,L(Y)]^n_1$, and we obtain
the convergence in probability of  $\sqrt{n} ([Y^+,L(Y)]^n_1,[Y^-,L(Y)]^n_1)$
to $(0,0)$.

We use Lemma~\ref{qvb} for dealing with the third summand:
\begin{multline}
    \label{convL}
 \sqrt{n} [L(Y)]_1^n
=\sqrt{n}\left(\frac{2\s_+\s_-}{\s_++\s_-}\right)^2 [L(X)]_1^n
\\
\convp[n\to\infty]
\frac{4\sqrt{2}}{3\sqrt{\pi}} \left(\frac{2\s_+\s_-}{\s_++\s_-}\right)^2 L_1(X) =
\frac{4\sqrt{2}}{3\sqrt{\pi}} \left(\frac{2\s_+\s_-}{\s_++\s_-}\right) L_1(Y).
\end{multline}
We obtain, using \eqref{scl},
\begin{multline}\label{cc}
\sqrt{n}\left(
\begin{pmatrix}[Y^+]_1^n \\ [Y^-]_1^n\end{pmatrix}-
\begin{pmatrix}\langle\xi\rangle_1 \\ \langle\eta\rangle_1)\end{pmatrix}
\right)
\\
\convsl[n\to\infty]
\begin{pmatrix}
\sqrt{2} \int_0^1 \s^2_+ \mathbf{1}(Y_s > 0)\vd\bar{B}_s \\
\sqrt{2} \int_0^1 \s^2_- \mathbf{1}(Y_s < 0) \vd\bar{B}_s
\end{pmatrix}
-
\begin{pmatrix}
    \vphantom{\int_0^1}  1\\
    \vphantom{\int_0^1}  1
\end{pmatrix}
\frac{\sqrt{2}}{3\sqrt{\pi}} \left(\frac{2\s_+\s_-}{\s_++\s_-}\right) L_1(Y).
\end{multline}
We write now
\begin{multline*}
\begin{pmatrix}
 (\hat{\s}^n_+)^2-\s_+^2 \\
 (\hat{\s}^n_-)^2-\s_-^2
\end{pmatrix}
=
\begin{pmatrix}
    \dfrac{[Y^+]^n_1 -  \s_+^2 \bar{Q}^n_1(Y,+) }{ \bar{Q}^n_1(Y,+) } 
    \\
   \dfrac{[Y^-]^n_1 -  \s_-^2 \bar{Q}^n_1(Y,-) }{ \bar{Q}^n_1(Y,-) } 
\end{pmatrix}
\\
=
\begin{pmatrix}
    \dfrac{[Y^+]^n_1 -  \langle\xi\rangle_1 }{ \bar{Q}^n_1(Y,+) } 
 \\
   \dfrac{[Y^-]^n_1 -  \langle\eta\rangle_1 }{ \bar{Q}^n_1(Y,-) } 
\end{pmatrix}
   +
\begin{pmatrix}
      \dfrac{ \s_+^2(Q^+_1 - \bar{Q}^n_1(Y,+)) }{ \bar{Q}^n_1(Y,+) } 
       \\
      \dfrac{\s_-^2 (1-Q^+_1 -  \bar{Q}^n_1(Y,-)) }{ \bar{Q}^n_1(Y,-) }
\end{pmatrix}.
\end{multline*}
Recall that $\bar{Q}^n_1(Y,+)$ and $\bar{Q}^n_1(Y,-)$ converge
almost surely to $Q^+_1(Y)$ and $Q^-_1(Y)=1-Q^+_1$. Besides, 
$0<Q^+_1<1$ a.s., because $Y_0=0$. Therefore, from Theorem \ref{ot},
\[
\sqrt{n}\begin{pmatrix} \dfrac{ \s_+^2(Q^+_1 - \bar{Q}^n_1(Y,+)) }{ \bar{Q}^n_1(Y,+) } 
    \\
  \dfrac{\s_-^2 (1-Q^+_1 -  \bar{Q}^n_1(Y,-)) }{ \bar{Q}^n_1(Y,-) } \end{pmatrix}
\convp[n\to\infty] \begin{pmatrix} 0 \\ 0 \end{pmatrix}.
\]
Using again \eqref{scl} and \eqref{cc},
\begin{multline}
\sqrt{n}
\begin{pmatrix}
\hat{\s}^n)^2 -  \s_+^2\\
\hat{\s}^n)^2 - \s_-^2  
\end{pmatrix}
\convsl[n\to\infty]
\begin{pmatrix}
\dfrac{\sqrt{2}\s^2_+}{Q^+_1}  \int_0^1 \mathbf{1}(Y_s > 0)\vd\bar{B}_s 
\\
\dfrac{\sqrt{2}\s^2_-}{1-Q^+_1} \int_0^1 \mathbf{1}(Y_s < 0) \vd\bar{B}_s
\end{pmatrix}
-\begin{pmatrix}\dfrac{1}{Q^+_1}\\
    \dfrac{1}{1-Q^+_1}
\end{pmatrix}
 \frac{\sqrt{2}}{3\sqrt{\pi}} \left(\frac{2\s_+\s_-}{\s_++\s_-}\right) L_1(Y).
\end{multline}
The statement is now proved, but we would like to get a more explicit expression for the law of the limit random variable. Recall $Q^+_t(Y)=Q^+_t(X)$.
From Corollary 1.2 in \cite{abtww}, standard computations give that the joint density of $(L_t(X),Q^+_t)$ is, 
for $b>0$, $\tau\in[0, t]$:
\[
p_{L_t(X),Q^+_t}(b,\tau)=\frac{(1-\th^2)b}{4 \pi\tau^{3/2}(t-\tau)^{3/2}} \exp\left(-\frac{(1+\th)^2 b^2 }{8 \tau }-\frac{(1-\th)^2 b^2}{8 (t-\tau)}\right).
\]
We set now 
\[
Z_t=\frac{L_t(X)}{4}\sqrt{\frac{(1+\th)^2}{Q^+_t}+\frac{(1-\th)^2}{t- Q^+_t}}.
\]
Changing variable in the integration, the 
joint density of $(Z_t,Q^+_t)$ is 
\begin{equation}
    \label{j1}
p_{Z_t,Q^+_t}(x,\tau)=
\left(x \exp\left(-\frac{x^2 }{2}\right)\right)\left(\frac{1}{\pi\tau^{1/2}(t-\tau)^{1/2}} \frac{1-\th^2}{(1+\th)^2(t-\tau)+(1-\th)^2\tau } \right).
\end{equation}
We also find the joint density of $(Z_t,t-Q^+_t)$ as
\begin{equation}
    \label{j2}
p_{Z_t,t-Q^+_t}(x,\tau)=
\left(x \exp\left(-\frac{x^2 }{2}\right)\right)\left(\frac{1}{\pi\tau^{1/2}(t-\tau)^{1/2}} \frac{1-\th^2}{(1+\th)^2\tau +(1-\th)^2(t-\tau) } \right).
\end{equation}
As we can factorize $p_{Z_t,Q^+_t}(x,\tau)=p_{Z_t}(x) p_{Q^+_t}(\tau)$, $Z_t$
and $Q^+_t$ are independent and their laws are explicit.
In particular from \eqref{ma}, for $t=1$, $p_{Z_1}(x)= x \exp(-\frac{x^2 }{2})$,
\[
\begin{split}
p_{Q^+_1}(\tau)&=
\frac{1}{\pi\tau^{1/2}(1-\tau)^{1/2}}\times \frac{1-\th^2}{(1+\th)^2(t-\tau)+(1-\th)^2\tau } \\
& =\frac{1}{\pi\tau^{1/2}(1-\tau)^{1/2}}\times \frac{\s_+/\s_-}{1-(1-(\s_+/\s_-)^2)\tau},
\end{split}
\]
and
\[
\begin{split}
p_{1-Q^+_1}(\tau)&=
\frac{1}{\pi\tau^{1/2}(1-\tau)^{1/2}}\times \frac{1-\th^2}{(1+\th)^2\tau+(1-\th)^2(1-\tau) } \\
& =\frac{1}{\pi\tau^{1/2}(1-\tau)^{1/2}}\times \frac{\s_-/\s_+}{1-(1-(\s_-/\s_+)^2)\tau}.
\end{split}
\]
Let now $\L$ be a random variable with the same law of $Q^+_1$,
and let $\xi$ be an independent exponential random variable of parameter $1$. From \eqref{j1}, \eqref{j2}
\[
    \begin{pmatrix}
\frac{1 }{Q^+_1}  
\\
\frac{1 }{1-Q^+_1}  
\end{pmatrix}
 L_1(X)
\eqlaw \frac{4 \xi}{\sqrt{(1-\L)(1+\th)^2+\L(1-\th)^2 }}
 \begin{pmatrix}
\sqrt{\frac{1-\L}{\L}}
\\
\sqrt{\frac{\L}{1-\L}}
\end{pmatrix}
{=}
 \frac{2 (\s_++\s_-)\xi}{\sqrt{(1-\L)\s_-^2+\L\s_+^2 }}
 \begin{pmatrix}
\sqrt{\frac{1-\L}{\L}}
\\
\sqrt{\frac{\L}{1-\L}}
\end{pmatrix}.
\]
Moreover, 
\[
    \begin{pmatrix}
\frac{\sqrt{2}\s_+^2 }{Q^+_1}  \int_0^1  \mathbf{1}(Y_s > 0)\vd\bar{B}_s
\\
\frac{\sqrt{2}\s_-^2}{1-Q^+_1} \int_0^1  \mathbf{1}(Y_s < 0)\vd\bar{B}_s 
\end{pmatrix}
\eqlaw
\begin{pmatrix}
\frac{\sqrt{2}\s_+^2 }{\sqrt{\L}} \mathcal{N}_1
\\
\frac{\sqrt{2}\s_-^2 }{\sqrt{1-\L}} \mathcal{N}_2
\end{pmatrix},
\]
where $\mathcal{N}_1,\mathcal{N}_2$ are standard Gaussian random variables independent of $\xi$ and $\L$.
Therefore the limit law has the expression given in the statement.
\end{proof}

\begin{proof}[Proof of Theorem \ref{maintheorem2}]
Using \eqref{tanaka}, we obtain
\[
[Y^+,Y]^n_1=\left[\xi+\frac{1}{2}L(Y),Y\right]^n_1=
[\xi,Y]^n_1+\frac{1}{2}[L(Y),Y]^n_1.
\]
From the Central Limit Theorem for martingales \cite[Theorem~5.4.2]{jp}
and $\xi_t+\eta_t=Y_t$,
\[
\begin{split}
\sqrt{n}
\begin{pmatrix}
 [\xi,Y]^n_1
\\ \relax
[\eta,Y]^n_1
\end{pmatrix}
&=
\sqrt{n}
\begin{pmatrix}
 [\xi,\xi]^n_1-[\xi,\eta]^n_1
\\ \relax
[\eta,\xi]^n_1-[\eta,\eta]^n_1
\end{pmatrix}
\\
& \convsl[n\to\infty]  
\sqrt{2} \int_0^1
\begin{pmatrix} 
 \s_+^2\mathbf{1}(Y_s > 0)
 \\
- \s_-^2\mathbf{1}(Y_s < 0)
\end{pmatrix}
\vd\bar{B}_s
=
\sqrt{2} \int_0^1
\begin{pmatrix} 
 \s_+^2\mathbf{1}(Y_s > 0)
 \\
\s_-^2\mathbf{1}(Y_s < 0)
\end{pmatrix}
\vd\tilde{B}_s,
\end{split}
\]
where $\tilde{B}$ is another BM independent of the filtration of $W$. Both $W$ and $\tilde{B}$
are defined 
on an extension $(\tilde{\Omega},\tilde{\mathcal{F}},\tilde{\PR})$ of $(\Omega,\mathcal{F},\PR)$ with $\PR'=\PR\otimes\PR'$
where $\PR'$ carries the BM $\tilde{B}$.
Moreover, 
\[
[L(Y),Y]^n_1=[L(Y),Y^+]^n_1-[L(Y),Y^-]^n_1=
\frac{2\s_+^2\s_-}{\s_++\s_-} [X^+,L(X)]^n_1
-\frac{2\s_+\s_-^2}{\s_++\s_-} [X^-,L(X)]^n_1
\]
so $\sqrt{n}[L(Y),Y]\convp 0$ because of Lemma \ref{qcs}. 
Finally,
\begin{equation*}
\sqrt{n}
\begin{pmatrix}
 [Y^+,Y]^n_1
\\ \relax
[Y^-,Y]^n_1
\end{pmatrix}
= 
\sqrt{2} \int_0^1
\begin{pmatrix} 
 \s_+^2\mathbf{1}(Y_s > 0)
 \\
 \s_-^2\mathbf{1}(Y_s < 0)
 \end{pmatrix}
 \vd\bar{B}_s.
\end{equation*}
This is the analogous of \eqref{cc} in the proof of Theorem \ref{maintheorem}. From now on the proof follows as in Theorem \ref{maintheorem}, but without the local time part.
\end{proof}

\begin{remark}\label{remaprlt}
We look for the origin of the asymptotic bias present in $\hat{\s}^n$, but not in $m^n$.
Consider first the difference between the approximation of quadratic variation used in the two different estimators:
\[
[Y^+,Y]_1^n=[Y^+,Y^+]_1^n-[Y^+,Y^-]_1^n = [Y^+]_1^n-[Y^+,Y^-]_1^n.
\]
From \eqref{tanaka},
\begin{multline*}
[Y^+,Y^-]_1^n= -[\xi,\eta]_1^n+\frac{1}{2}[\xi,L(Y)]_1^n-\frac{1}{2}[L(Y),\eta]_1^n+\frac{1}{4}[L(Y)]_1^n \\
= -[\xi,\eta]_1^n+\frac{1}{2}[Y^+,L(Y)]_1^n+
\frac{1}{2}[L(Y),Y^-]_1^n
-\frac{1}{4}[L(Y)]_1^n.
\end{multline*}
From the central limit theorem for martingales \cite[Theorem 5.4.2]{jp},
\[
\sqrt{n}[\xi,\eta]^n_1\convsl[n\to\infty] \sqrt{2} \int_0^1 \s_+\mathbf{1}(Y_s > 0) \s_- \mathbf{1}(Y_s < 0) \vd\bar{B}_s= 
0.
\]
Since $\sqrt{n}[Y^\pm,L(Y)]^n_1$ converges in probability to $0$, 
using \eqref{convL} we obtain 
\begin{equation}\label{aprlt}
\sqrt{n}[Y^+,Y^-]^n_1\convsl[n\to\infty] \frac{2\sqrt{2}}{3\sqrt{\pi}}\frac{\s_+\s_-}{\s_++\s_-}L_1(Y).
\end{equation}
We then see that the asymptotic bias in $\hat{\sigma}^2_\pm$ is related to the 
bracket $[Y^+,Y^-]^n_1$.
\end{remark}

%%%%%%%%%%%%%%%%%%%%%%%%%%%%%%%%%%%%%%%%%%%%%%%%%%%%%%%%%%%%%%%%%%%%%%
%%%%%%%%%%%%%%%%%%%%%%%%%%%%%%%%%%%%%%%%%%%%%%%%%%%%%%%%%%%%%%%%%%%%%%
%%%%%%%%%%%%%%%%%%%%%%%%%%%%%%%%%%%%%%%%%%%%%%%%%%%%%%%%%%%%%%%%%%%%%%

\subsection{Proof of Corollary \ref{girsanov}: adding a drift term via Girsanov Theorem}

Let us start with a remark on the stability of the stable convergence under a Girsanov transform. 

\begin{lemma}\label{changeprobstable}
    For two probability spaces $(\Omega,\mathcal{F},\PR)$ and $(\Omega',\mathcal{F}',\PR')$,
    let us define an extension $(\tilde{\Omega},\tilde{\mathcal{F}},\tilde{\PR})$ by 
of $(\Omega,\mathcal{F},\PR)$ of the form 
\begin{equation*}
    \tilde{\Omega}=\Omega\times \Omega',\ \tilde{\mathcal{F}}=\mathcal{F}\otimes\mathcal{F}'
    \text{ and }
    \tilde{\PR}=\PR\otimes\PR'.
\end{equation*}
Assume that $(\Omega,\mathcal{F},\PR)$ and $(\Omega',\mathcal{F}',\PR')$ carry respectively Brownian motions $W$ and $W'$ 
with natural (completed) filtrations $\mathscr{F}=(\mathcal{F}_t)_{t>0}$
and $\mathscr{F}'=(\mathcal{F}'_t)_{t>0}$. Assume also that $W$ and $W'$ are independent.

On $(\Omega,\mathcal{F},\PR)$, let $G$ be an exponential $\mathscr{F}$-martingale which is uniformly integrable. 
Let $\QR$ be the measure such that $\left.\frac{\vd\QR}{\vd\PR}\right|_{\mathcal{F}_t}=G_t$.

Suppose now that a sequence $Z_n$ on $(\Omega,\mathcal{F},\PR)$ of $\cF_T$-measurable random 
variables converges stably to a random variable $Z=AW'_B$ on the extension $(\tilde{\Omega},\tilde{\mathcal{F}},\tilde{\PR})$
of $(\Omega,\mathcal{F},\PR)$ where $A$ and $B$ are $\cF_T$-random variables on $(\Omega,\mathcal{F},\PR)$.

Then $Z_n$ converges stably to $Z=AW'_B$ on $(\tilde{\Omega},\tilde{\mathcal{F}},\QR\otimes\PR')$
where $W'$ is a Brownian motion independent from $A$ and $B$ (the laws of $A$ and $B$ are of course changed).
\end{lemma}

\begin{proof}
Let us write $\tilde{\QR}=\QR\otimes\PR'$. The Girsanov weight $G_T$ is $\cF_T$-measurable and integrable 
with respect to $\PR$. Hence, it is easily shown that for any bounded, $\cF_T$-measurable random variable $Y$
and any bounded, continuous function $f$, $\mathbb{E}[GYf(Z_n)]$ converges to $\mathbb{E}_{\tilde{\PR}}[GYf(Z)]$.

Under $\tilde{\QR}$, $\langle W',W\rangle =0$ as $W$ and $W'$ are independent and the bracket does
not change under a Girsanov transform. This implies that $W'$ is still a Brownian motion
under $\tilde{\QR}$. Hence the result.
\end{proof}

\begin{proof}[Proof of Corollary \ref{girsanov}]
Let $\xi$ be solution to $\vd \xi_t=\sigma(\xi_t)\vd W_t$
with an underlying Brownian motion $W$ on $(\Omega,\mathcal{F},\PR)$.
We denote by $(\mathcal{G}_t)_{t\geq 0}$ the filtration of $W$.
Thus, $\xi$ is an OBM.

The Girsanov theorem is still valid for discontinuous coefficients \cite{legall}. 
Let us set 
\[
G_t=\exp \left(\int_0^t \frac{b(\xi_s)}{\s(\xi_s)}\vd W_s
    -\frac{1}{2}\int_0^t \left(\frac{b(\xi_s)}{\s(\xi_s)}\right)^2 \vd s\right).
\]
Since $b$ is bounded, we define a new measure $\QR$ by $\left.\frac{\vd \QR}{\vd \PR}\right|_{\mathcal{G}_t}=G_t$.
Under $\QR$, the process $\xi$ is solution to $\vd \xi_t=\sigma(\xi_t)\vd \tilde{W}_t+b(\xi_t)\vd t$
for a Brownian motion $\tilde{W}_t=W_t-\int_0^t b(\xi_s)\sigma(\xi_s)^{-1}\vd s$, $t\geq 0$.

Theorems \ref{maintheorem} and \ref{maintheorem2} hold for $\xi$ under $\PR$.
Therefore, Lemma~\ref{changeprobstable} applies here. Thus, 
\[
 \sqrt{n}
\begin{pmatrix}
(m^n_+(\xi))^2 - \s_+^2\\
(m^n_-(\xi))^2 - \s_-^2
\end{pmatrix}
\xrightarrow[n\to\infty]{\tilde{\QR}\text{-sl}}
\begin{pmatrix}
\frac{\sqrt{2}\s_+^2 }{Q^+_1(\xi)}  \int_0^1  \mathbf{1}(\xi_s > 0)\vd\bar{B}_s
\\
\frac{\sqrt{2}\s_-^2}{1-Q^+_1(\xi)} \int_0^1  \mathbf{1}(\xi_s < 0)\vd\bar{B}_s 
\end{pmatrix},
\]
where $\bar{B}$ is a BM independent of $W$ and $\tilde{W}$ also under $\tilde{\QR}$.
\end{proof}

%%%%%%%%%%%%%%%%%%%%%%%%%%%%%%%%%%%%%%%%%%%%%%%%%%%%%%%%%%%%%%%%%%%%%%
%%%%%%%%%%%%%%%%%%%%%%%%%%%%%%%%%%%%%%%%%%%%%%%%%%%%%%%%%%%%%%%%%%%%%%
%%%%%%%%%%%%%%%%%%%%%%%%%%%%%%%%%%%%%%%%%%%%%%%%%%%%%%%%%%%%%%%%%%%%%%

\section{Oscillating Random Walk}\label{sectionORW}

In \cite{kw} the OBM is constructed also as a limit of discrete processes, called Oscillating Random Walks (ORW),
analogously to how the BM is constructed as a limit of Random Walks. 
The aim of this section is to examplify the phenomena of dependence on the occupation
of the limit law, in a simpler framework and with a non-technical proof.

We define the ORW as the following process. Fix $0<p,q\leq 1$. For $k\in\N$, we introduce the following random variables:
\begin{alignat*}{2}
U_k \text{ iid},\quad &\PR(U_k=1)=\PR(U_k=-1)=\frac{q}{2}, &\quad \PR(U_k=0)=1-q,\\
V_k \text{ iid},\quad & \PR(V_k=1)=\PR(V_k=-1)=\frac{p}{2}, &\quad \PR(V_k=0)=1-p,\\
Z_k \text{ iid},\quad & \PR(Z_k=1)=\PR(Z_k=-1)=\frac{1}{2}.
\end{alignat*}
Now we set $Y^*_0=0$ and
\[
Y^*_{k+1}= \begin{cases}
Y^*_{k}+U_{k+1} & \text{if } Y^*_k> 0, \\
Y^*_{k}+V_{k+1} & \text{if } Y^*_{k}< 0, \\
Y^*_{k}+Z_{k+1} & \text{if } Y^*_{k}= 0.
\end{cases}
\] 
We consider the re-normalized process 
\[
Y^n_t=n^{-1/2} Y^*_{[nt]}.
\]
For all $K>0$, we have the following convergence:
\[
\sup_{0\leq t \leq K} |Y^n_t-Y_t| \convp[n\to\infty] 0.
\]
The convergence in probability holds if the processes $Y^n$ are constructed as in \cite{stone}, and~$Y$ is an OBM of parameters $\s_+^2=q$, $\s_-^2=p$. This means that in this setting we have $0< \s_-,\s_+\leq 1$, but we do not loose in generality because we can always re-scale time and space. 

In this appendix, 
we recover from the observations of $Y^n$ for some large $n$ an estimator for the parameters of the OBM.

We set $\b^n=\#\{k\in\N,\,k\leq n :Y^n_{k/n}> 0\}$, $\a^n=\#\{k\in\N,\,k\leq n:Y^n_{k/n}>0, Y^n_{(k+1)/n}\neq Y^n_{k/n} \}$, and introduce the following estimator of $q=\sigma_+^2$:
\begin{equation}\label{estorw}
\hat{q}^n=\frac{ \a^n}{\b^n}.
\end{equation}

\begin{theorem}
Let $\hat{q}^n$ be the estimator defined above. The following convergence holds:
\[
\sqrt{n} \left(\hat{q}^n-q\right) \convl[n\to\infty] \mathcal{N} \sqrt{ \frac{q(1-q)}{\L}},
\]
where $\L$ follows the law in \eqref{ma}, $\mathcal{N}$ is a standard Gaussian and they are independent.
\end{theorem}

\begin{proof}
When $Y^n_{k/n}>0$,  $Y^n_{(k+1)/n}\neq Y^n_{k/n}$ with probability $q$, and $Y^n_{(k+1)/n}= Y^n_{k/n}$ with probability $1-q$. We can compute the log-likelihood and maximize it as in the statistics of Binomial variables, finding that the maximum likelihood estimator for $q$ is $\hat{q}^n$ in \eqref{estorw}. In~\cite{kw} it is proved that
\[
\frac{\#\{k\leq n :Y_k\geq 0\}}{n} \convl[n\to\infty] \L,
\]
where $\L$ follows the law in \eqref{ma}. This easily implies
\begin{equation}\label{cabm}
\frac{\b^n}{n} \convl[n\to\infty] \L.
\end{equation}
Conditioning to $\b^n$, we have that $\a^n$ follows is a binomial distribution with parameters $q,\b^n$.
We write the event
\[
\left\{ \sqrt{n} \left(\hat{q}^n-q\right)\leq x \right\}
=
\left\{ \sqrt{ \frac{\b^n}{q(1-q)}}\left(\hat{q}^n-q\right)\leq x \sqrt{ \frac{\b^n}{n q(1-q)}} \right\}.
\]
From Berry-Essen inequality \cite{berry41a,esseen42a}, we have
\[
\left| \PR\Prb*{ \sqrt{n} (\hat{q}^n-q )\leq x \given \b^n}
-
\Phi\left(x \sqrt{ \frac{\b^n}{nq(1-q)}} \right)
\right|
\leq
C_q/\sqrt{\b^n},
\]
for some constant $C_q$. Now, from \eqref{cabm} and Portmanteau Lemma,
\[
\E\Prb*{ \Phi\left(x \sqrt{ \frac{\b^n}{nq(1-q)}}  \right)}
 \xrightarrow[n\to\infty]{}
\E\Prb*{ \Phi\left(x \sqrt{ \frac{\L}{q(1-q)}}  \right)}.
 \]
Moreover,
$\E [C_q/\sqrt{\b^n} ]\rightarrow 0$.
Recalling
\begin{equation*}
\PR\Prb*{ \sqrt{n} (\hat{q}^n-q)\leq x }=
\E\Prb*{ \PR\Prb*{ \sqrt{n} \left(\hat{q}^n-q\right)\leq x\given \b^n}},
\end{equation*}
we obtain the following convergence
\[
    \PR\Prb*{ \sqrt{n} \left(\hat{q}^n-q\right)\leq x }\xrightarrow[n\to\infty]{}
    \E\Prb*{ \Phi\left(x \sqrt{ \frac{\L}{q(1-q)}}  \right)},
\]
which implies the statement.
\end{proof}

%%%%%%%%%%%%%%%%%%%%%%%%%%%%%%%%%%%%%%%%%%%%%%%%%%%%%%%%%%%%%%%%%%%%%%
%%%%%%%%%%%%%%%%%%%%%%%%%%%%%%%%%%%%%%%%%%%%%%%%%%%%%%%%%%%%%%%%%%%%%%
%%%%%%%%%%%%%%%%%%%%%%%%%%%%%%%%%%%%%%%%%%%%%%%%%%%%%%%%%%%%%%%%%%%%%%

\section{Empirical evidence}\label{numerics}

In this section we implement estimators $\hat{\s}^n,m^n$ and use them on simulated data. For doing so, 
we reduce the OBM \eqref{obm} to a SBM \eqref{sbm}, and we simulate it through the simulation 
method given in \cite{lejay_2012}. This method gives the successive positions $\{X_{k \Delta t}\}_{k\geq 0}$
of the SBM, hence the OBM, on a time grid of size $\Delta t$. 

Recall Remark \ref{scaling}, in particular estimators $m^{n,T}_\pm,\hat{\sigma}^{n,T}_\pm$, for which we have central limit theorems with the same limit laws of \eqref{eq:limit:biased},  \eqref{eq:limit:unbiased}. 
We use the parameters:  $T=5$,  $\Delta t=0.01$ (thus $n=500$), $\sigma_-=0.5$, $\sigma_+=2$  (so that $\theta=-0.48$
in \eqref{sbm}, pushing the process to the negative side).
In Figure~\ref{fig:1}, we plot the density of 
\begin{equation*}
M^n_\pm:= \sqrt{n}((m^{n,T}_\pm)^2-\s_\pm^2)\text{ and }
S^n_\pm:= \sqrt{n}((\hat{\mathsf{\sigma}}^{n,T}_\pm)^2-\s_\pm^2)
\end{equation*}
for $N$ realizations of these estimators (meaning the simulation of $N=10\,000$ paths of the OBM). 
Their empirical densities are compared with the ones of 
\begin{equation*}
M_\pm:=\frac{\sqrt{2}\sigma_\pm^2\mathcal{N}}{\sqrt{\Lambda}}
\text{ and }
\begin{pmatrix}
S_+\\
S_-
\end{pmatrix}
:=
\begin{pmatrix}
\frac{\sqrt{2}\s_+^2 }{\sqrt{\L}} 
\left(\mathcal{N}_1
-
 \frac{8}{3\sqrt{\pi}} \frac{1}{r+1}
 \frac{\xi \sqrt{1-\L}
}{\sqrt{(1-\L)+\L r^2 }}\right)
\\
\frac{\sqrt{2}\s_-^2 }{\sqrt{1-\L}} 
\left(\mathcal{N}_2
-
 \frac{8}{3\sqrt{\pi}} \frac{1}{1/r+1}
 \frac{\xi \sqrt{\L}
}{\sqrt{\L+(1-\L)/r^2 }}\right)
\end{pmatrix},
\end{equation*}
given in \eqref{eq:limit:unbiased} and \eqref{eq:limit:biased}, with $\mathcal{N}\sim N(0,1)$ and $\xi\sim\exp(1)$.
The densities of $M$ and $S$ (which do not depend on $T$) are obtained by simulation. 
The occupation time $\Lambda$ is simulated by inverting its distribution function  \cite{lamperti,lejay11}:
\begin{equation*}
    \Lambda\eqlaw \frac{\sigma_-^2 V}{\sigma_-^2 V+\sigma_+^2(1-V)}\text{ with }V=\sin^2\left(\frac{U\pi}{2}\right),
    \ U\text{ uniform on }[0,1).
\end{equation*}

\begin{figure}[ht!]
    \begin{center}
	\includegraphics{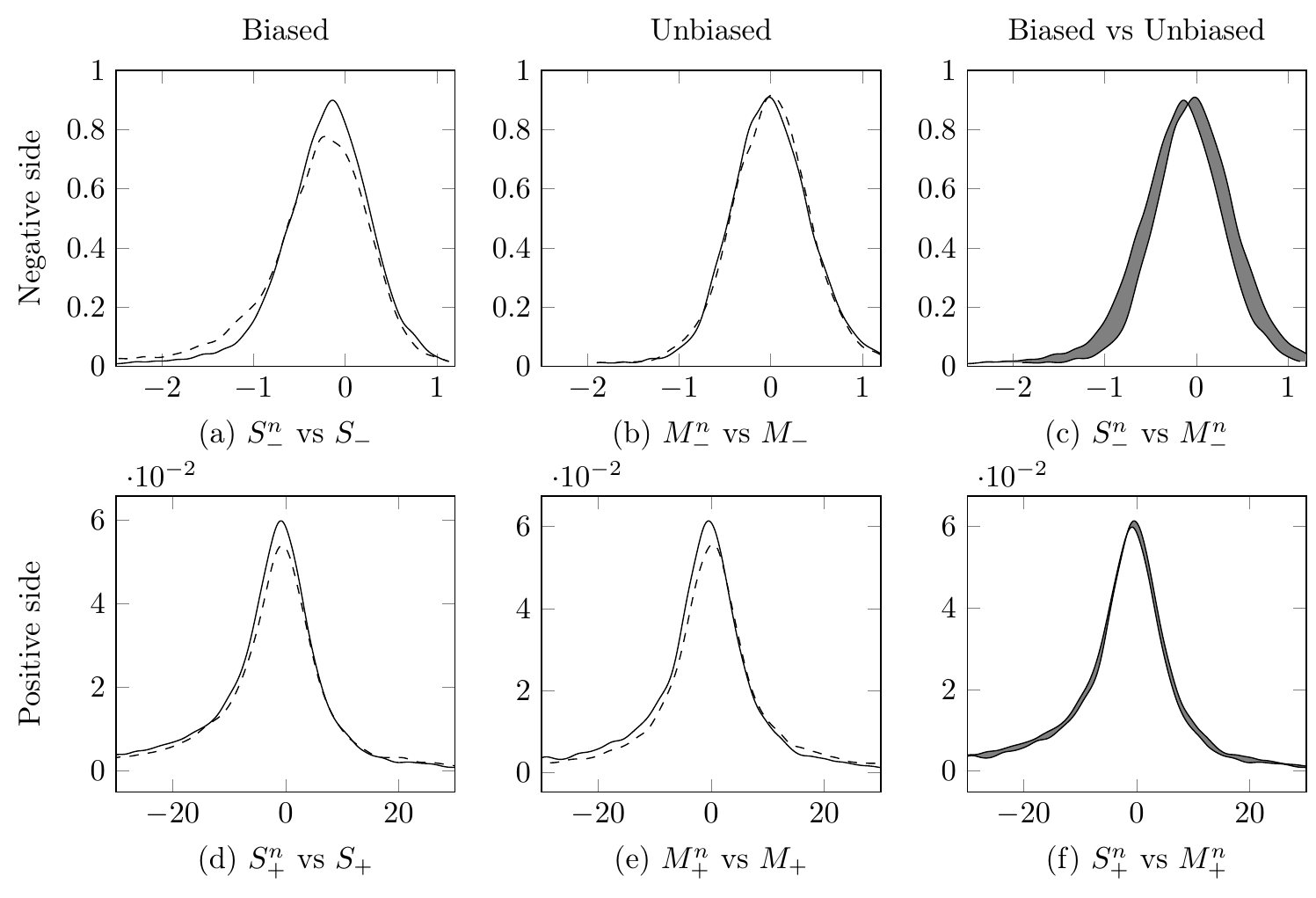}
    \caption{\label{fig:1} Densities of:\\
(a), (d) normalized error of the estimators $M^n_\pm$ (solid line) and theoretical limits $M_\pm$ (dashed lines);\\
(b), (e) normalized error of the estimators $S^n_\pm$ (solid line) and theoretical limits $S_\pm$ (dashed lines);\\
 (c), (f) normalized error of the estimators $M^n_\pm$ and normalized error of the estimators  $S^n_\pm$. }
\end{center}
\end{figure}
We see that the limit distribution on the positive side has a larger variance than the one in the negative side. This is due to the sample size, proportional to the occupation time, which is on average larger on the side where the diffusion coefficient is smaller. We also obtain a good agreement of the normalized empirical error with the prediction of the central limit theorem. On the pictures on the right, we observe the difference between biased and non-biased estimator; the grey area is the effect given by the shift to the left of the distribution, caused by the local time term. This shift is more visible when the diffusion coefficient is smaller.

We have also checked that 
$\sqrt{n}\frac{[Y+,Y^-]_T^n}{ T}$
has a distribution close to the one of 
$\frac{L_T(Y)}{\sqrt{T}}= \frac{\sqrt{2}}{3\sqrt{\pi}}\left(\frac{2\s_+\s_-}{\s_++\s_-}\right)^2 L_1(X)$, 
which is straightforward since the density of $L_1(X)$ is known (for this, we use P.~Lévy's identity which
relates the local time to the supremum of the Brownian motion whose density is explicitly known). 
The agreement is good.

Finally, the same simulation work can be done using the random walks defined in Section~\ref{sectionORW}
using the simple approach. Again, the numerical results are in good agreements with the theory, although 
some instabilities appear due to the fact that the occupation time may take small values with high probability.

\bibliographystyle{plain}

\bibliography{biblio-estimator-obm.bib}

\end{document}